\documentclass[12pt]{amsart}

\usepackage{amsmath, amssymb, amsfonts, amsthm}
\usepackage{geometry}
\usepackage{mathdots}
\usepackage[mathscr]{euscript}

\newtheorem{remark}{Remark}

\DeclareMathOperator{\SL}{SL}
\DeclareMathOperator{\GL}{GL}

\newcommand{\A}{\mathbb{A}}
\newcommand{\R}{\mathbb{R}}
\newcommand{\Z}{\mathbb{Z}}
\newcommand{\C}{\mathbb{C}}
\newcommand{\Q}{\mathbb{Q}}

\newcommand{\bsm}{\left(\begin{smallmatrix}}
\newcommand{\esm}{\end{smallmatrix}\right)}
\newcommand{\bpm}{\begin{pmatrix}}
\newcommand{\epm}{\end{pmatrix}}

\usepackage[usenames,dvipsnames]{xcolor}

\newtheorem{theorem}{Theorem}[section]
\newtheorem{lemma}[theorem]{Lemma}
\newtheorem{corollary}[theorem]{Corollary}
\newtheorem{proposition}[theorem]{Proposition}

\renewcommand{\Re}{\operatorname{Re}}

\newcommand{\inv}{^{-1}}
\renewcommand{\d}{\,\mathrm{d}}
\newcommand{\sums}{\sideset{}{^*}\sum}

\DeclareMathOperator{\diag}{diag}
\renewcommand{\th}{\textsuperscript{th}}
\renewcommand{\vec}[1]{\mathbf{#1}}
\renewcommand{\S}{\mathscr{S}}
\DeclareMathOperator{\U}{U}

\newcommand{\ea}[2]{e_\alpha \left(\frac{#1}{#2}\right)}
\newcommand{\eb}[2]{e_\beta \left(\frac{#1}{#2}\right)}

\usepackage{cancel, mathtools, amscd, comment}

\title{Parametrization of Kloosterman Sets and $\SL_3$-Kloosterman Sums}
\author{Eren Mehmet K{\i}ral\textsuperscript{1}, Maki Nakasuji\textsuperscript{2}}
\numberwithin{equation}{section}
\thanks{1. Supported by Grant-in-Aid for JSPS Research Fellow 18F18326}
\thanks{2. Supported by Grant-in-Aid for Scientific Research (C) 18K03223}

\address{Department of Information and Communication Science, Faculty of Science, \\ Sophia University, \\
7-1 Kioi-cho, Chiyoda-ku, Tokyo, 102-8554, Japan}

\email{erenmehmetkiral@protonmail.com}
\email{nakasuji@sophia.ac.jp}



\begin{document}

\begin{abstract}
We stratify the $\SL_3$ big cell Kloosterman sets using the reduced word decomposition of the Weyl group element, inspired by the Bott-Samelson factorization. Thus the $\SL_3$ long word Kloosterman sum is decomposed into finer parts, and we write it as a finite sum of a product of two classical Kloosterman sums. The fine Kloosterman sums end up being the correct pieces to consider in the Bruggeman-Kuznetsov trace formula on the congruence subgroup $\Gamma_0(N)\subseteq \SL_3(\Z)$. Another application is a new explicit formula, expressing the triple divisor sum function in terms of a double Dirichlet series of exponential sums, generalizing Ramanujan's formula.   
\end{abstract}
\maketitle

\section{Introduction}
\subsection{Statement of results} 

A Kloosterman set of the ``\emph{big cell}'' in $\SL_3$ is
\begin{equation} \label{eq:OmegaCoarseDefinition}
	\Omega = \Omega_{w_0}(c_1, c_2) = \left\{A \in \SL_3(\Z): A \in \U_3(\Z) w_0 \bsm c_1 && \\ &\frac{c_2}{c_1} & \\ &&\frac{1}{c_2} \esm  \U_3(\Z) \right\},
\end{equation}
where $\U_3$ is the group of $3\times 3$ unipotent matrices and $c_1, c_2$ are nonzero integers. The long word $\SL_3$ Kloosterman sum with modulus $\vec{c} = (c_1,c_2)$ can be described as a sum over $\Omega_{w_0}(c_1,c_2)$. In this paper we give a finer decomposition of \eqref{eq:OmegaCoarseDefinition} via the sets $\Omega(d_1,d_2,f)$ defined in  \eqref{eq:FineKloostermanSetGL3}. This finer decomposition is insipired by a reduced word decomposition of $w_0$ and the subsequent Bott-Samelson factorization of flag varieties.

This stratification gives a decomposition of the long word $\SL_3$ Kloosterman sum into what we call \emph{fine Kloosterman sums}. In order to distinguish it, we denote by script $\mathscr{S}_w$: 
\begin{equation} \label{eq:FineCoarseKloostermanDecomposition}
S_{w_0} (\vec{m},\vec{n}; (c_1,c_2)) = \sum_{f | \gcd(c_1, c_2) } \S_{w_0} \left(\vec{m}, \vec{n}; \frac{c_1}{f}, \frac{c_2}{f}, f\right).
\end{equation}
We parametrize $\Omega(d_1,d_2,f)$, thus obtaining nice expressions for  $\S_{w_0} (\vec{m}, \vec{n}; d_1,d_2,f)$.

Using the parametrization from Corollary \ref{cor:FineKloostermanCellBijection} we give the following presentation of the long word $\SL_3$ Kloosterman sum as a sum of a product of two classical $\SL_2$-Kloosterman sums in  Theorem \ref{thm:KloostermanEvaluation}.  
\begin{equation}\label{eq:FineKloostermanEvaluation}
  \S_{w_0} (\vec{m}, \vec{n}; d_1,d_2,f) = f \sum_{\substack{x, y \pmod{f} \\ xy \equiv 1 \pmod{f}\\ m_2d_2 + n_2d_1y \equiv 0 \pmod{f}}}S(n_1,N(y);d_1) S(m_1,M(x);d_2).
\end{equation}
Here  $N = N(y) := (m_2d_2 + n_2d_1 y)/f$ and $M  = M(x) := (m_2d_2x + n_2 d_1)/f$.
When $f = 1$ this simplifies to $S(n_1, m_2d_2; d_1) S(m_1,n_2d_1; d_2)$, or for the more general case of $\gcd(f,d_1d_2) = 1$, see Proposition \ref{prop:coprimeSimplification}.

Note that \eqref{eq:FineCoarseKloostermanDecomposition} and  \eqref{eq:FineKloostermanEvaluation} together give us that $S_{w_0}(\vec{m}, \vec{n}, (c_1,c_2) )$ lies in a real algebraic number field. In fact it lies in a compositum of fields of the form $\Q(\cos\big(\tfrac{2\pi }{p^k}\big))$, for various primes $p$ and integers $k$.

Another application is the following explicit formula for the triple divisor sum.
\begin{equation}\label{eq:divisorFormula}
	\sigma_{1-s_1, 1-s_2}(1,n) = \zeta(s_1) \zeta(s_2) \zeta(s_1 + s_2 -1) \sum_{d_1, d_2 = 1}^\infty \frac{\mu(d_1)}{d_1^{s_1}d_2^{s_2}} \sum_{f | n d_1} \frac{c_{f}(n) c_{d_2} \left(\frac{n d_1}{f}\right)}{f^{s_1 + s_2 -1}}.
\end{equation}
Here $c_q(n) = \sum_{\substack{u\pmod{q}\\ \gcd(u,q) =1}} e(un/q)$ is the classical Ramanujan sum modulo $q$ and $\sigma_{\nu_1,\nu_2}(1,n)$ is defined in \eqref{eq:divisorSumDefinition}, and $\Re(s_1), \Re(s_2)>1$ ensures convergence of the right hand side. This is a direct generalization of the Ramanujan formula 
\begin{equation}\label{eq:ramanujanFormula}
\sigma_{1-s}(n) = \zeta(s) \sum_{q = 1}^\infty \frac{S(0,n;q)}{q^s}  \qquad \text{ with } \Re(s) >1,
\end{equation}
and the right hand side of \eqref{eq:divisorFormula} is 
\[
	\zeta(s_1) \zeta(s_2) \zeta(s_1 + s_2-1) \sum_{c_1, c_2 >0} \frac{S_{w_0}( (0,0),(1,n) ; (c_1,c_2)) }{c_1^{s_1} c_2^{s_2}}, 
\]
written according to \eqref{eq:FineCoarseKloostermanDecomposition} and \eqref{eq:FineKloostermanEvaluation}. In short, just as the $\SL_2$ case, the $\SL_3$ Ramanujan formula is a consequence of evaluating the $(1,n)$\th\ Fourier coefficient of an Eisenstein series as a double Dirichlet series. However in Section \ref{sec:RamanujanFormula} we also give a direct and elementary proof of this formula.

Our work is motivated by aesthetic considerations, believing that a beautiful expression for a Kloosterman sum would increase its recognizability when encountered elsewhere in nature, and its comprehensibility. This goal was shared by \cite{KiralYoungKloosterman}, explicitizing the rank one Kloosterman sums with respect to various cusp-pairs. We hope that similarly this work encourages the use of the explicit form of the Kloosterman sum, and leads to deeper results, better bounds and discovery of new identities for moments of $L$-functions.

$\SL_3$ Kloosterman sums have been previously been calculated. In the seminal work of \cite{BumpFriedbergGoldfeld}, the authors used Pl\"ucker coordinates to parametrize the double cosets of the Bruhat cells of $\SL_3$. 
This formulation has recently has been used in myriad applications, especially in the context of $\SL_3$ Kuznetsov trace formula, see \cite{BlomerApplications}, \cite{BlomerButtcaneMaga}, \cite{GoldfeldKontorovich}, \cite{YoungBilinearSL3}, \cite{BlomerButtcaneGlobalDecomposition}.

In most of these works, the authors are content to provide an upper bound for the Kloosterman sums when they arise. The upper bounds have been given by Larsen for the Weyl group elements $w$ of length two in the appendix of \cite{BumpFriedbergGoldfeld}, and by Stevens for the long word in \cite{Stevens}. In Section \ref{sec:KloostermanBound} we provide an upper bound for the long word Kloosterman sum, that is essentially as strong as the one by Stevens \cite{Stevens}, and stronger in its $\vec{n}, \vec{m}$ dependency. Here, our use of the word essentially refers to the fact that our bound differs only by a factor $\tau(\gcd(c_1,c_2))$ with $\tau(n) = \sum_{d|n} 1$.

Finally we should note that our stratification encodes the level structure in a simple manner. The fine Kloosterman sums appearing in Bruggeman-Kuznetsov trace formula for the congruence  group $\Gamma_0(N)$ are exactly those with $N |f$. This is a simple condition, which implies $N | c_1$ and $N|c_2$ in the notation of \eqref{eq:FineCoarseKloostermanDecomposition}, but is not conversely implied by it. For the convenience of the reader, we write down the $\Gamma_0(N) \subseteq \SL_3(\Z)$ Bruggeman-Kuznetsov formula using fine Kloosterman sums in Section \ref{sec:Kuznetsov}.

\subsection{The historical background and the previous literature}
The exponential sum
\[
	S(m,n;c) := \sum_{\substack{a, d \pmod{c} \\ ad\equiv 1 \pmod{c}}} e\left(\frac{ma}{c} + \frac{nd}{c} \right),
\]
is called the classical Kloosterman sum, first introduced by H. D. Kloosterman in \cite{KloostermanOriginal} in the context of bounding the error term arising from the circle method of G.H. Hardy, J. E. Littlewood and S. Ramanujan \cite{HardyLittlewoodWaring, HardyRamanujanPartition}. Here we use the notation $e(z) = e^{2\pi i z}$, for $z \in \C$.

A second context in which the Kloosterman sums appear involves exponential sums over $\gamma  =\bsm a & b \\c & d\esm \in \SL_2(\Z)$,
for example in the computation of the Fourier coefficients of the classical Poincar\'{e} series.

In this second context a connection to the spectral theory of automorphic forms is forged. The connection is due to the presence of Kloosterman sums on the \emph{geometric side} of the Petersson and Bruggeman-Kuznetsov trace formulas.\footnote{In \cite{KloostermanOriginal}, the use of $\theta$-function as a generating function for the set of square integers also provided a connection to modular forms.} 
The spectral theory of automorphic forms is central estimating $L$-function moments, obtaining hyperbolic equidistribution results, quantum ergodicity on hyperbolic spaces, as well as being fascinating area of investigation in and of itself. For $\SL_2$ automorphic forms, the Bruggeman-Kuznetsov/Petersson trace formulas have been the workhorse of virtually any result in analytic number theory concerning a family of automorphic forms and $L$-functions.

Given the central importance of Kloosterman sums in the rank 1 theory, attention has also turned to the explicit calculation of both the Kloosterman sums associated to $\SL_r$ such as in \cite{Friedberg1987}, \cite{Stevens}, and the integral transforms in the higher rank Bruggeman-Kuznetsov formula, see \cite{buttcaneSpectral}. For work on the Kuznetsov formula on $\SL_3$ see also the work of Buttcane \cite{buttcane2019arithmetic} and Blomer, Buttcane and Maga \cite{BlomerButtcaneMaga} where they include the level structure. In this paper our focus is on the explicit computation of the $\SL_3$ Kloosterman sums.

The work of \cite{Friedberg1987} notices the general rank $r$ hyperkloosterman sum as the Kloosterman sum associated to the cyclic element $(12\cdots r)$ of the Weyl group $\operatorname{Sym}_r$ of $\SL_r$. Our work shares the use of the exterior algebra in determining the coordinates of various factorizations. 

\subsection{Method of Proof}

Our calculation is heavily influenced by, but does not directly use, the Bott-Samelson decomposition of a flag variety. We saw this approach first in the work of Brubaker and Friedberg in \cite{BrubakerFriedberg}, in the context of calculating the Fourier coefficients of metaplectic Eisenstein series. Given a Weyl group element $w$ and $w = s_{\alpha_1} \cdots s_{\alpha_\ell}$ a reduced word decomposition of $w$, we can write 
\begin{equation}\label{eq:BwBreducedWordDecomposition}
	BwB = (Bs_{\alpha_1}B) (Bs_{\alpha_2} B) \cdots (Bs_{\alpha_\ell} B).
\end{equation}
In fact we can accomplish this in quite a generality, see \cite{GarrettBuildings}. Our approach in this work is to find the necessary conditions such that given an $A \in BwB \cap \SL_r(\Z)$, we can write
$$\iota_{\alpha_1}(\gamma_1) \cdots \iota_{\alpha_\ell}(\gamma_\ell) \in \Gamma_\infty A \Gamma_\infty,$$ where $\gamma_i \in \SL_2$, in the big cell, i.e.\ with a nonzero lower left entry. It would be simplest if we could independently choose each $\gamma_i \in \U_2(\Z) \backslash B \bsm &-1\\ 1& \esm B \cap \SL_2(\Z)/ \U_2(\Z)$. However, the reality is subtler. In this paper, we work out the various integrality conditions and the interdependencies among the $\gamma_i$'s. 

For $\SL_3(\R)$, and $w = w_0= s_\alpha s_\beta s_\alpha$ the long word Weyl group element, the lower left entries of $\gamma_i$ form a triple of nonzero integers $d_2$, $f$ and $d_1$ respectively. Stratifying the $\U_3(\Z)$--double cosets of $Bw_0B \cap\SL_3(\Z)$ according to these triples of integers turns out to give an easily parametrizable set. We then write the long word $\SL_3(\Z)$ Kloosterman sum separately as sums over these strata. The end result is \eqref{eq:FineKloostermanEvaluation}.

\subsection{Discussion}

Historically Kloosterman introduced his sum \cite{KloostermanOriginal}, in the context of the circle method applied to the sum of four squares. The problem had no Bruhat decomposition in sight. This coincidence can be exploited to find yet another connection between automorphic forms and the study of integer points on algebraic surfaces \cite{Steiner}. An understandable formula for a $\SL_3$ (or higher rank) Kloosterman sum may allow researchers to recognize Kloosterman sums when they see them in their research. Thus, for the researchers working on more complicated problems involving the circle method, the exponential sums they obtain may signal to them that the there may be a connection to higher rank automorphic forms. 

We expect that our detailed investigation into the structure of the higher rank Kloosterman sums will also lead to a finer understanding of higher rank automorphic forms. As an example, recently there has been a flurry of activity in spectral reciprocity formulas, see \cite{BlomerLiMillerSpectral}, \cite{BlomerKhanTwistedFourth},  \cite{AndersenKiral}, \cite{Zacharias}, \cite{PetrowTwistedMotohashi} and of course the seminal work of Motohashi \cite{Motohashi93}. These are formulas where both sides contain a moment, or a twisted moment of a family of $L$-functions with possibly some correction terms.  One way to obtain these results is to pass from either side, perhaps via a trace formula, to a sum of exponential sums and connect these exponential sums. At this step precise and practical knowledge of the exponential sums is necessary. Great insight is to be gained from finding connections between various moments. 

In a more straightforward way we also expect our results to be useful in the spectral theory of higher rank automorphic forms. Even though there have been deep results concerning higher rank automorphic forms, see \cite{LiGL3GL2}, \cite{BlomerLiMillerSpectral}, \cite{LiYoung}, these have all used the $\SL_2$ spectral theory and Bruggeman Kuznetsov formula. The notable exceptions to these are \cite{BlomerApplications}, \cite{BlomerButtcaneMaga}, and \cite{YoungBilinearSL3} where the sums are over $\SL_3$ automorphic forms.  Except for the last one, the works have not gone into the \emph{guts} of the geometric exponential sums, but rather used bounds.

Also we can use the methods of this paper to consider the metaplectic case. As noted in \cite{BrubakerFriedberg} and \cite{BBFAnnalsEisenstein} the decomposition of $A = \prod_{i=1}^r \iota_{\alpha_i}(\bsm a_i&b_i\\ c_i&d_i\esm)$ helps us easily write the Kubota symbol $\kappa(A)$ using $n$\th\ power residue symbols $\big(\tfrac{d_i}{c_i}\big)_n$ multiplicatively.

In \cite[Chapter 5.4, p.215]{MotohashiBook} Motohashi has noted that just as the Ramanujan formula for the divisor function was used in an essential manner in obtaining the spectral formula for the fourth moment of the Riemann zeta function in \cite{Motohashi93}, its generalization for the triple divisor function forms a connection between the sixth moment of the Riemann zeta function and the $\SL_3(\Z)$ theory, and continues to emphasize that \emph{``\ldots it is highly desirable to have an honest extension to $\SL(3, \Z)$ of the theory developed in Chapters 1-3''}. Bump in \cite{BumpGL3} has found such a formula, as Motohashi notes, even though this establishes the connection to the $\SL_3(\Z)$ theory, the exact form of the divisor formula was not amenable to concrete calculations.

Notice that for $s_1 = s_2 = 1$ the left hand side of \eqref{eq:divisorFormula} is the triple divisor function $\tau_3(n) = \sum_{n_1n_2n_3 = n} 1$.
Our formula gives a way to \emph{expand} $\tau_3(n)$ into a double Dirichlet series of exponential sums, which hopefully can be useful in separating additive terms that appear in shifted convolution sums such as $\sum_{n \ll X} \tau_3(n) \tau_3(n+h)$. 

\vspace{3mm}

{\bf Acknowledgements.} We would like to thank RIKEN iTHEMS for providing the first author with office space. We also would like to thank Matthew Young, Valentin Blomer and Daniel Bump for useful comments. We would especially like to thank Solomon Friedberg and Jack Buttcane for catching a significant error in an earlier draft.

\section{Notation and Background} \label{sec:notation}

The group of $r\times r$ unipotent matrices, i.e.\ upper triangular matrices with $1$'s on the diagonal entries, is denoted by $\U_r$. Let $T_r$ be the torus subgroup, i.e. diagonal matrices in $\SL_r$. We will drop the $r$ from the notation when we fix its value throughout a section. Call $\Gamma_\infty = \U(\Z) $.

We refer the reader to \cite[Chapter 18]{Springer} for background on root systems of reductive algebraic groups. For our purposes we will only be concerned with classical semisimple Lie groups of type $A$, i.e. $\SL_r$. given a positive root $\alpha$ we let $\iota_\alpha: \SL_2(\R) \to \SL_r(\R)$ be the canonical morphism onto the subgroup $G_\alpha\subseteq \SL_r(\R)$ which is a rank one subgroup with root system spanned by $\alpha$.

In the case of $\SL_3$, with the simple roots $\alpha = (1, -1, 0)$ and $\beta = (0, 1, -1)$, and $\gamma = \bsm a & b \\c & d \esm \in \GL_2$,
\[
	\iota_{\alpha}\left(\gamma \right) = \bpm a& b&\\ c&d&\\ &&1\epm, \qquad 	
	\iota_{\beta}\left(\gamma \right) = \bpm 1&&\\ &a&b\\ &c&d \epm,
	\qquad   
    \iota_{\alpha + \beta}\left(\gamma  \right) = \bpm a& &b\\ &1&\\ c&&d \epm.  
\]
Throughout the paper, unwritten coordinates in a matrix are assumed to be $0$. 

Going back to the general case, for any root $\alpha$, 
\begin{align*}
	e_\alpha(x) = \iota_\alpha\left(\bpm 1&x \\ 0&1 \epm\right) , &\quad 	e_{-\alpha}(x) = \iota_\alpha\left(\bpm 1&0 \\ x&1 \epm\right), \\
	h_\alpha(c)  = \iota_\alpha\left(\bpm c&0 \\ 0&1/c \epm\right), &\quad s_\alpha= \iota_\alpha\left(\bpm 0&-1 \\ 1&0 \epm\right), 
\end{align*}
and note that $s_\alpha$ corresponds to the Weyl element reflecting $\alpha$ to $-\alpha$. 

The Weyl group of $\SL_r$ can be identified with the symmetric group on $r$ letters. The simple roots correspond to adjacent transpositions, i.e. $(i, i + 1)$. By abuse of notation we will denote a Weyl group element by the same letter $w$ and use it in both as an element in $\operatorname{Sym}_r$ and the matrix $\prod_{j=1}^\ell \iota_{\alpha_j}\left(\bsm 0&-1 \\ 1 & 0 \esm\right)$ where the order and the terms of the product are taken as in the word decomposition $w = s_{\alpha_1}s_{\alpha_2} \cdots s_{\alpha_\ell}$. If this is a reduced word decomposition, i.e. $\alpha_i$ are simple roots and this product minimal among such representations, then $\ell = \ell(w)$ is called the length of $w$. 

Given two $r \times r$ matrices $A$ and $B$ we will write $A\sim B$ to mean that there are $u_1, u_2 \in \Gamma_\infty$ such that $A = u_1 B u_2$.

Given an $r \times r$ square matrix if $I, J \subseteq \{ 1,2, \ldots, r\}$ of equal size we denote by $M_{I,J}$ the minor determinant obtained by taking the entries in the columns in $I$ and rows in the set $J$. In case $|I| = |J| = 1$ these are simply the entires of the matrix $A$, and we denote them by $A_{ij}$ and in case we have $I = \{i\}^c, J= \{j\}^c$, i.e. complements of a singleton then we denote the minor via $M_{ij}$. 

Let $V$ be an $r$ dimensional vector space, with $\vec{e}_1, \ldots, \vec{e}_r$ as standard basis vectors. Given an element $A \in \GL_r$ the action of $A$ on elements of the $k$-fold wedge product are defined as
\[
	A(v_1 \wedge v_2 \wedge \cdots \wedge v_k) = (Av_1) \wedge (Av_2) \wedge \cdots \wedge (Av_k).
\]
For subsets $I = \{i_1 < i_2 < \ldots < i_k\} \subseteq \{1, \ldots, r\}$ the vectors $\vec{e}_I:= \vec{e}_{i_1} \wedge \cdots \wedge \vec{e}_{i_k}$, form a basis of $\bigwedge^k V$. The action of $A$ is calculated explicitly via the minors as,
\[
	A \vec{e}_I = \sum_{\substack{J \subseteq \{1, \ldots, r\}\\ |J|=k}} M_{I, J} \vec{e}_J.
\] 
Writing $\vec{e}_J^* := \vec{e}_{j_1}^*\wedge \vec{e}_{j_2}^* \wedge \cdots \wedge \vec{e}_{j_k}^* \in (\bigwedge^k V)^* \cong \bigwedge^k V^*$, where $\vec{e}_1^*, \ldots, \vec{e}_r^*$ are the dual standard basis elements of $V^*$, we can also write $M_{I,J} =\langle \vec{e}_{J}^*, A\vec{e}_I \rangle.$

\section{Coordinates of the Bruhat decomposition} \label{sec:BruhatCoordinates}

In this section given $A \in  \U w T \U$ with decomposition written as 
\begin{equation}\label{eq:BwBdecomposition}
	A = \bpm 1& a_{12} & a_{13} & \cdots & a_{1r}\\ 
	&1&a_{23}& \cdots & a_{2r} \\
	&&\ddots &  &\vdots\\
	&&&&1 \epm 
	w \bpm t_1&&&\\&t_2& \\ && \ddots & \\ &&&t_r \epm 
	\bpm 1& b_{12} & b_{13} & \cdots & b_{1r}\\ 
	&1&b_{23}& \cdots & b_{2r} \\
	&&1& \vdots &\vdots\\
	&&&&1 \epm ,
\end{equation}  
we understand the coordinates $a_{ij}, b_{ij}$ and $t_i$, in terms of the entries of $A$. We use the action of $A$ on the exterior algebra $\bigoplus_{k = 0}^r \bigwedge^k V$ in order to write the parameters in terms of various ratios of minors of $A$. If we assume $A$ is an integral matrix, then this also gives information on the integrality of these coordinates, since all minors $M_{I,J} \in \Z$.

Let $$w_0 = \left(\begin{matrix} 1 & 2& 3 & \cdots & (r-1) & r \\ \downarrow & \downarrow & \downarrow & \cdots &\downarrow & \downarrow \\ r & (r-1) & (r-2) & \cdots  &2 & 1 \end{matrix}\right).$$ If we write a simple transposition as $s_i = (i(i+1))$ and further write $c_k$ for the cyclic permutation $c_k = s_{k-1} s_{k-2} \cdots s_2 s_1 = (k(k-1)\cdots 321)$.The reduced word decomposition of the long word using simple transpositions is given by,
\[
	w_0 = c_2 c_3 \ldots c_r = s_1 (s_2 s_1) (s_3 s_2 s_1) \cdots (s_{r-1} \cdots s_2 s_1).
\]
The long word is the only permutation where for every $1\leq i < j \leq r$ we have $w(i) > w(j)$.

In terms of the matrix representation in $\SL_r$, as described in Section \ref{sec:notation}
\[
	w_0 = \bpm &&&&&& \pm 1 \\ &&&&&\iddots & \\ &&&&-1&&\\ &&&1&&&&\\ &-1&&&&&&\\ 1&&&&&&& \epm_{r \times r}.
\]

\begin{proposition} \label{prop:BruhatCoordinatesViaExterior}
	Let $A \in \GL_r$ be an element of the big cell and let $a_{ij}, b_{ij}$ and $t_k$ be defined via the Bruhat decomposition as in \eqref{eq:BwBdecomposition} with $w = w_0$. The entires are determined via minors of $A$ as
	\[
		t_1 \ldots t_k = \langle \vec{e}^*_{r-k+1, r-k + 2 \ldots, r - 1, r}, A \vec{e}_{1,2,\ldots, k} \rangle,
	\]
	and the upper triangular matrices have entries determined by
	\[
		a_{ij} =  \frac{\langle \vec{e}^*_{i, j+1, j + 2, \ldots, r}, A \vec{e}_{1, 2, \ldots, r-j + 1} \rangle  }{t_1t_2 \cdots t_{r-j + 1}}, 
	\]
	and
	\[
		b_{ij} =\frac{\langle \vec{e}^*_{r- i + 1, \ldots, r}, A \vec{e}_{1, 2, \ldots, i -1, j}\rangle }{ t_1 t_2 \cdots t_i }.
	\]
\end{proposition}
Let us express the identity \eqref{eq:BwBdecomposition} as $A = u_L w_0 t u_R$ and then in matrix form we see that
\[
	u_L = \left(\begin{matrix} 1 & \frac{\langle \vec{e}^*_{1,3, \ldots, r}, A \vec{e}_{1,2, \ldots, r-1} \rangle}{t_1\cdots t_{r-1}}  & \frac{\langle \vec{e}^*_{1, 4, 5, \ldots, r}, A \vec{e}_{1,2, \cdots, r - 2} \rangle  }{t_1 \cdots t_{r-2} }  & \cdots  & \frac{\langle \vec{e}^*_{1,r}, A\vec{e}_{1,2}\rangle }{t_1t_2} & \frac{\langle \vec{e}^*_{1}, A \vec{e}_1\rangle  }{t_1}\\ 
	0 & 1 & \frac{\langle\vec{e}^*_{2,4,5, \ldots, r}, A\vec{e}_{1,2, \ldots, r-2} \rangle }{t_1 \cdots t_{r-2}}  & \cdots & \frac{\langle \vec{e}^*_{2, r}, A \vec{e}_{1,2}\rangle }{t_1 t_2 }  & \frac{\langle \vec{e}^*_2, A \vec{e}_1\rangle }{t_1} \\
	0&0&1 &\ddots & \vdots & \vdots \\
	0&0&0&\ddots & \vdots &\vdots \\
	0& 0 & 0 & \cdots & 1& \frac{\langle \vec{e}^*_{r-1} , A\vec{e}_1 \rangle }{t_1}\\
	0 & 0& 0 &  \cdots & 0 &1 	
	 \end{matrix}\right).
\]
In words, the structure of the entries of $u_L$ is as follows: All the entries are given by fractions. The denominators of the entries in the $k$\textsuperscript{th} column from the right are given by the $k\times k$ minor in the lower left corner of $A$. 
The numerator of the entry at the $i$\th\ row is another $k\times k$ minor, where the top row of the minor in the denominator is replaced by the $i$\th\ row of $A$.

The right unipotent factor is given as
\[
	u_R = \left(\begin{matrix} 1 & \frac{\langle \vec{e}^*_{r}, A \vec{e}_{2} \rangle}{t_1}  & \frac{\langle \vec{e}^*_{r}, A \vec{e}_{3} \rangle  }{t_1 }  & \cdots  & \frac{\langle \vec{e}^*_{r}, A\vec{e}_{r-1}\rangle }{t_1} & \frac{\langle \vec{e}^*_{r}, A \vec{e}_r\rangle  }{t_1}\\ 
	0 & 1 & \frac{\langle\vec{e}^*_{r-1, r}, A\vec{e}_{1,3} \rangle }{t_1 t_2}  & \cdots & \frac{\langle \vec{e}^*_{r-1, r}, A \vec{e}_{1,r-1}\rangle }{t_1 t_2 }  & \frac{\langle \vec{e}^*_{r-1,r}, A \vec{e}_{1,r}\rangle }{t_1t_2} \\
	0&0&1 &\ddots & \vdots & \vdots \\
	0&0&0&\ddots & \vdots &\vdots \\
	0& 0 & 0 & \cdots & 1& \frac{\langle \vec{e}^*_{2,\ldots, r} , A\vec{e}_{1, \ldots, r-2,r} \rangle }{t_1\cdots t_{r-1}}\\
	0 & 0& 0 &  \cdots & 0 &1 	
	 \end{matrix}\right).
\]  
In words, the entries $b_{ij}$ of $u_R$ are given as follows: They are given by fractions where the denominator of the entries in the  $i$\th\ row are the $i\times i$ determinant of the lower left corner. The numerator of the fraction of the entry at the $j$\th\ column is then obtained by replacing the $i$\th\ (the last) column of this $i \times i$ minor by entries from the $j$\th\ column (and the last $i$ rows) of $A$.
\begin{proof}
Let us apply $A$ to $\vec{e}_{1, \ldots, k}$. These particular wedge products is fixed by the group $\U$ for each $k = 1, 2, \ldots, r$. We thus have from \eqref{eq:BwBdecomposition} with $w = w_0$,
\[
	A \vec{e}_{1, \ldots, k} = t_1, \cdots t_k  u_L w_0 \vec{e}_{1, \ldots, k} = t_1 \cdots t_k u_L \vec{e}_{r-k + 1, r-k + 2, \ldots, r}.
\] 
Note that we can apply $u_L$ in parts. 
\begin{multline*}
	u_L \vec{e}_{r-k + 1, r-k + 2, \ldots, r} = u_L \vec{e}_{r - k + 1} \wedge (u_L \vec{e}_{r - k + 2, r - k + 3, \ldots, r}) \\
	=   \vec{e}_{r - k + 1} \wedge u_L \vec{e}_{r - k + 2, \ldots, r} + \sum_{i = 1}^{r - k} a_{i (r - k + 1)} \vec{e}_{i} \wedge u_L 
	\vec{e}_{r- k + 2, \ldots, r}. 
\end{multline*}
Also since $u_L$ is unipotent, we will have $\langle  \vec{e}_{r - k + 2, \ldots , r}^*, u_L\vec{e}_{r - k + 2, \ldots , r} \rangle = 1$. 

Therefore indeed 
\begin{align*}
	t_1t_2 \cdots t_k &= \langle \vec{e}_{r- k + 1}^*\wedge \vec{e}^*_{r - k + 2, \ldots, r}, A\vec{e}_{1,\ldots, k}  \rangle, 	\\
	t_1t_2 \cdots t_k  a_{i (r -k + 1)} &=  \langle \vec{e}^*_{i} \wedge \vec{e}^*_{r-k + 2, \ldots, r} ,  A \vec{e}_{1, \ldots, k} \rangle ,
\end{align*} 
for all $i = 1, 2, \ldots r - k$ . Plugging in $j = r - k + 1$ yields the result.

As for the matrix $u_R$, we consider $A \vec{e}_{1, \ldots i-1, j}$. Firstly $u_R \vec{e}_{1, \ldots ,i-1, j} = b_{ij} \vec{e}_{1,\ldots, i} + \cdots.$ Continuing the calculation,
\[
	A \vec{e}_{j, \ldots, r} = t_1\cdots t_i b_{ij} u_L \left(\vec{e}_{r - i + 1, \ldots, r} + \cdots\right).
\]
Since $u_L$ is a unipotent matrix the only way to obtain $\vec{e}_{r -i + 1, \ldots, r} $ after a multiplication by $u_L$ is if the input includes this basis vector. 

Then we get 
\[
	t_1t_2 \cdots t_k  b_{ij} =  \langle \vec{e}^*_{r-i + 1, \ldots, r},    A \vec{e}_{1, \ldots, i-1, j} \rangle, 
\] 
for every $j = i + 1, \ldots, r$.
\end{proof}

For a general Weyl group element $w$, we can still explicitly evaluate the diagonal elements $t_i$. The rest of the chapter can be found in \cite{Friedberg1987} and \cite{Stevens}, but we include it here for completeness. 

\begin{lemma}[Proposition 3.1 from \cite{Friedberg1987}]
	Let $w \in W$ be any Weyl group element and $A \in BwB \cap \SL_r(\Z)$ with coordinates as in \eqref{eq:BwBdecomposition}. Then
	\[
		t_1 \cdots t_k = \pm \langle \vec{e}_{j_1, \ldots, j_k}^*, A \vec{e}_{1,2, \ldots, k} \rangle ,
	\]
	where $\{j_1 < j_2< \cdots < j_k\} = \{w(1), w(2), \ldots, w(k) \}$.
	
	This implies that $\diag(t_1, \ldots, t_r) = \diag(c_1, c_2/c_1, c_3/c_2, \ldots, c_{r-1}/c_{r-2}, 1/c_{r-1})$ with all $c_i$ nonzero integers. Furthermore if $(\ell_1, \ldots, \ell_k) \geq (j_1, \ldots, j_k)$ with the inequality coordinatewise, with at least one strict inequality, then the minor $\langle \vec{e}_{\ell_1, \ldots, \ell_k} , A\vec{e}_{1, 2, \ldots, k} \rangle$ vanishes. 
\end{lemma}
\begin{proof}
	Firstly note that given a unipotent matrix $u$, and $\vec{e}_{i_1, \ldots, i_k}$, then
	\begin{equation}\label{eq:unipotentActionWedge}
		u \vec{e}_{i_1, \ldots, i_k} = \vec{e}_{i_1, \ldots, i_k} + \sum_{\substack{j_1 \leq i_1, \ldots, j_k \leq i_k \\ \text{with one inequality strict} \\ \text{and } j_1 < j_2< \cdots <j_k }} c_{I,J} \vec{e}_{j_1, \ldots j_k}.
	\end{equation}
	for some constants $c_{I,J}$ depending on the entries of $u$.
	 
	Therefore group $\U_r(\Z)$ fixes the vectors $\vec{e}_{1,\ldots, k}$. The action of the diagonal element on $\vec{e}_{1,\ldots, k}$ is by multiplying with $t_1\cdots t_k$. Then the action of $w$ on $\vec{e}_{1, \ldots k} = \pm \vec{e}_{w(1)} \wedge \vec{e}_{w(2)} \wedge \cdots \wedge \vec{e}_{w(k)}$. The ambiguity in the sign comes from the negative $-1$ entries in $w$, which may cancel out once the set $\{w(i): i = 1, \ldots k\}$ is reordered. Finally the action of $u_L$ on $\vec{e}_{j_1, j_2, \ldots j_k}$ will output the vector itself plus other terms as given by \eqref{eq:unipotentActionWedge}. The vanishing minors of the Bruhat cell can also be observed from this equation since those wedge products will not appear in the sum on the right.
\end{proof}

For a matrix  $A \in \U w T \U$, for $w \neq w_0$ we can still write $A = u_L(A) w t(A) u_R(A)$ as in \eqref{eq:BwBdecomposition}, however the entries of $u_L(A)$ or $u_R(A)$ are not uniquely determined.  

\begin{lemma}
Let $w \in S_r$ be a Weyl group element, and let us have $A \in \U w T \U$. Let us write $\U_w = (w\inv \U w) \cap \U$. Then
\[
	A = u_L w t u_R = u_L' w t' u_R'
\]
if and only if $t = t'$ and $u_R' u_R\inv \in \U_w$ and $u_L\inv u_L' \in \U_{w\inv}$. Note that $w\U_w w\inv = \U_{w\inv}$.
\end{lemma}

\begin{proof}
	One can see that by computing $\langle \vec{e}_{w(1), w(2), \ldots, w(k)}, A\vec{e}_{1,2, \ldots, k} \rangle$ on either side, we get $t_1 \cdots t_k$ and $t_1' \cdots t_k'$ for any $k = 1, \ldots, r$. This implies that $t = t'$. 
	
	Then denoting $wtw\inv = t^w$, we have $$w\inv \underbrace{(t^w)\inv ({u_L'}\inv u_L) t^w}_{\in \U} w = \underbrace{u_R' u_R\inv}_{\in \U}.$$ This equation implies that $\U_w u_R  = \U_w u_R'$. Similarly 
	\[
		 w  \underbrace{t (u_R {{u_R'}\inv}) t\inv}_{\in \U} w\inv = \underbrace{u_L\inv u_L'}_{\in \U},
	\]
	therefore $u_L \U_{w\inv} = u_L' \U_{w\inv}$. So we have that $u_L(A)$ is determined up to $\U/\U_{w\inv}$ and $u_R(A)$ is determined up to left $\U_w$ cosets, i.e. $\U_w \backslash \U$. 
\end{proof}

Because Weyl group elements on matrices by conjugation simply by permutation, the groups $\U_w$ can be explicitly calculated as
\begin{equation} \label{eq:Uwexplicit}
	\U_w = \{u = [u_{ij}]_{i,j} : u_{ii}= 1 \text{ and } u_{ij} = 0 \text{ unless } i < j \text{ and } w(i) < w(j)\},
\end{equation}
since $w A$ rearranges the rows of $A$ according to the permutation $w$ and $Aw$ rearranges the columns of $A$ according to the permutation $w\inv$. 

For $\SL_3$ we can calculate these groups simply $\U_I = \U$, $\U_{w_0} = \{I\}$ and
\begin{equation}\label{eq:unipotentIntersectionGroups}
\begin{split}
	\U_{s_{\alpha}} = \left\{\bpm 1& 0 &*\\ &1&*\\&&1 \epm \right\}, \qquad \qquad
	\U_{s_{\beta}} &= \left\{\bpm 1& * &* \\ & 1& 0 \\ &&1 \epm \right\},\\
	\U_{s_{\alpha}s_\beta} = \left\{\bpm 1& * & 0 \\ &1&0\\ &&1\epm \right\}, \quad \quad \quad
	\U_{s_\beta s_\alpha} &= \left\{\bpm 1& 0 &0\\ &1&* \\ &&1 \epm  \right\}.
\end{split}
\end{equation}

It is clear from \eqref{eq:Uwexplicit} that the opposite group $\U^w := w\inv \U^t w \cap \U$, also given as 
\begin{equation} \label{eq:oppositeUw}
	\U^w = \{u = [u_{ij}]_{i,j} : u_{ii} = 1 \text{ and } u_{ij} = 0 \text{ unless } i < j \text{ and } w (i) > w (j)\},
\end{equation}
satisfies $\U_w \cap \U^w = \{I_{r\times r}\}$ and $\U_w\U^w = \U^w \U_w = \U$ (also called the Zappa-Sz\'{e}p product), see \cite[Section 2]{Stevens} and \cite[Section 1]{Friedberg1987}. We can equate $\U/\U_{w\inv}$ with $\U^{w\inv}$ and $\U_w \backslash \U$ with $\U^w$.

In terms of the groups \eqref{eq:unipotentIntersectionGroups} these opposite groups are obtained by exchanging the $*$ and $0$ symbols above the diagonal.

What we can assume is that given a decomposition $A = u_L w t u_R$, we can decompose $u_R = u_R^+ u_R^-$, with $u_R^+ \in \U_w$ and $u_R^-\in \U^w$, and simply move the plus part across to the other side:
\begin{equation} \label{eq:BruhatLeftRightmoves}
	A = u_L \underbrace{ w \overbrace{(tu_R^+ t\inv)}^{\in \U_w} w\inv}_{\in U} w t u_R^-. 
\end{equation}

Combining this with the previous lemma we get the following result.
\begin{lemma} 
Let $w$ be a Weyl group element. Given a matrix $A \in BwB$, and $\U^w$ is defined by \eqref{eq:oppositeUw}, there is a unique decomposition $A = u_L wt u_R$ with $t$ diagonal,  $u_L \in \U$ and $u_R \in \U^w$.
\end{lemma}


\section{The Kloosterman sets and Kloosterman sums} \label{sec:KloostermanSetsSums}

For a vector $\vec{c} \in (\R^*)^{r-1}$ define, $t(\vec{c}) := \diag(c_1, c_2/c_1, c_3/c_2, \ldots, 1/c_{r-1})$ and  
\[
	\Omega_w(\vec{c})  := \left\{u_L w t(\vec{c}) u_R \in \SL_3(\Z): u_L, u_R \in \U \right\}. 
\]
Note that by the previous section we could equivalently also have taken $u_R \in \U^w$. We have seen that for a matrix $A \in \SL_r(\Z) \cap BwB$ the $c_i$ are integers given by minors of $A$, and they not changed upon multiplication by elements of $\U$ from either side.

Therefore we obtain the stratification into finite sets:
\[
	\U(\Z) \backslash BwB \cap \SL_r(\Z) / \U^w(\Z) = \bigcup_{\vec{c} \in \Z^{r-1} } \U(\Z)\backslash \Omega_{w} (\vec{c}) /\U^w(\Z).
\]

Let $\vec{n} = (n_1, n_2, \ldots, n_{r-1})\in \Z^{r-1}$ and define the additive character $\psi_{\vec{n}}$ as follows: Let $u$ be a unipotent matrix, where for $i<j$ its entries are denoted by $u_{i,j}$. Then 
\[
	\psi_{\vec{n}} (u) = e(n_1 u_{1,2} + n_2 u_{2,3} + \cdots + n_{r-1} u_{r-1, r}).
\]

For $\vec{m}, \vec{n} \in \Z^{r-1}$ we define the usual (coarse) Kloosterman sum  as
\begin{equation} \label{eq:CoarseKloostermanDefinition}
	S_{w}(\vec{m}, \vec{n}; \vec{c}) = \sum_{\substack{A \in \U(\Z)\backslash \Omega_{w}(\vec{c}) /\U^w(\Z) \\A = u_L w t(\vec{c}) u_R }} \psi_{\vec{m}}(u_L) \psi_{\vec{n}}(u_R).
\end{equation}

This is the standard definition of a Kloosterman sum, however we want to distinguish it from the finer decomposition we will consider. For $\SL_3$ and $w = w_0$ we will decompose the set into finer gradations, and consider the sums over those sets. 

Before closing off the section let us write the conditions necessary for this sum to be well defined. Since given an element of $\U_w$ we may move an element of $\U_w$ from the right hand side of a Bruhat decomposition to the left hand side of a Bruhat decomposition as in \eqref{eq:BruhatLeftRightmoves}, the two characters should agree. Therefore the term is well defined only if
\[
	\psi_{\vec{m}}(w(tut\inv) w\inv) = \psi_{\vec{n}}(u).
\]
for every $u \in \U_w$, and $t$ diagonal. 

\begin{proposition}
	Let $w$ be a Weyl group element. the sum given by \eqref{eq:CoarseKloostermanDefinition} is well defined if and only if \begin{enumerate}
	\item for all $k \in\{1,2, \ldots, r-1\}$ satisfying $w(k) < w(k + 1)$, we have $n_k = 0$ and in case $w(k + 1) \neq w(k) + 1$; and  
	\[
		\frac{c_{k}^2}{c_{k-1}c_{k + 1}} m_{w(k)} = n_{k},
	\]
	in case $w(k + 1) = w(k) + 1$. (Here $c_0$ and $c_r$ are understood as $1$.)
	\item for all $k \in \{1,2,\ldots, r-1\}$ satisfying $w\inv(k) < w\inv(k + 1)$, we have $m_k = 0$ in case $w\inv(k + 1)\neq w\inv(k) + 1$.
	\end{enumerate}
\end{proposition}

Let us write these conditions in the case of $\SL_3$.
\begin{equation}\label{eq:consistencyConditions}
\begin{tabular}{c|c|c|c|c}
$w $ &  $(123)$ & $(132)$ & $(12)$ & $(23)$ \\ \hline
\begin{minipage}{3.9cm} \begin{center}$S_w(\vec{m},\vec{n},\vec{c})$ as in\\ \eqref{eq:CoarseKloostermanDefinition}  is well defined  if \end{center}   \end{minipage} &  $n_1 =  \frac{c_1^2}{c_2}m_2$ & $n_2 = \frac{c_2^2}{c_1} m_1 $  & $n_2 = m_2 = 0$ & $n_1 = m_1 = 0$   
\end{tabular}
\end{equation}

\section{$\SL_3$ long word fine Kloosterman sums.}

\subsection{The fine Kloosterman set stratification}
Given $A = Bw_0B \cap \SL_3(\Z)$ we write
\[
	A = \left[\begin{matrix}
	A_{11} & A_{12} & A_{13} \\ 
	A_{21} & A_{22} & A_{23} \\
	A_{31} & A_{32} & A_{33}
	\end{matrix}\right].
\]

With coordinates as in the factorization \eqref{eq:BwBdecomposition} with $w = w_0$, the above calculations give us that
\[
	t_1 = A_{31} \in \Z-\{0\} \qquad \text{ and } \qquad t_1t_2 = \left| \begin{matrix} A_{21} & A_{22}\\ A_{31} & A_{32} \end{matrix}\right| \in \Z-\{0\},
\]
and also $c_1 = t_1$, $c_2 = t_1t_2$ with the notation in the beginning of Section \ref{sec:KloostermanSetsSums}.

\begin{lemma} \label{lem:gl3commongcd}
	Given $A \in \SL_3(\Z) \cap Bw_0B$ as above, we have $$\gcd(A_{31}, A_{32} )= \gcd\left(\left|\begin{matrix} A_{21}& A_{22} \\ A_{31} & A_{32} \end{matrix} \right|, \left|\begin{matrix} A_{11}& A_{12} \\ A_{31} & A_{32} \end{matrix}\right|\right).$$
\end{lemma}
	Call this common g.c.d. by $f$. Given such an $f$, we define $d_1, d_2 \in \Z$ via
\begin{equation}\label{eq:d1d2fDefinition}
	c_1 = A_{31} =  d_1f \qquad \text{ and } \qquad c_2 = \left| \begin{matrix} A_{21} & A_{22}\\ A_{31} &A_{32} \end{matrix}\right|  = d_2f.
\end{equation}

Given $d_1, d_2, f$ nonzero integers, with some abuse of notation, define, 
\begin{equation}\label{eq:FineKloostermanSetGL3}
	\Omega(d_1, d_2, f) := \left\{A \in \SL_3(\Z)| \gcd(A_{31}, A_{32}) = f, A_{31} = d_1f, M_{\{23\}, \{12\}} = d_2 f\right\}.
\end{equation}
These sets stratify the coarse Kloosterman set as follows,
\begin{equation}\label{eq:KloostermanSetFineDecomposition}
	\Omega_{w_0} (c_1,c_2) = \bigsqcup_{f | \gcd(c_1, c_2) } \Omega\left( \frac{c_1}{f}, \frac{c_2}{f}, f\right) . 
\end{equation}
The sets on the right hand side of this finer decomposition are invariant under the action of $\Gamma_\infty$ from both sides. Defining the \emph{fine Kloosterman sums} to be sums restricted to these sets we end up with
\begin{equation}\label{eq:FineKloostermanSum}
	\S_{w_0}(\vec{m},\vec{n}; d_1,d_2,f) = \sum_{\substack{A \in \Gamma_\infty \backslash \Omega(d_1,d_2,f)/\Gamma_\infty \\ A \in u_L w_0t(d_1f, d_2f ) u_R} } \psi_{\vec{m}} (u_L) \psi_{\vec{n}}(u_R).
\end{equation}
The (usual) coarse Kloosterman sum thus can be written as the sum $$S_{w_0} (\vec{m},\vec{n}; (c_1,c_2)) = \sum_{f | \gcd(c_1, c_2) } \S_{w_0} \left(\vec{m}, \vec{n}; \frac{c_1}{f}, \frac{c_2}{f}, f\right).$$
	
\begin{proof}[Proof of Lemma \ref{lem:gl3commongcd}]
	Let $f_1$ be the greatest common divisor of the left hand side, and $f_2$ of the right hand side. Simply by noting that both of the $2\times 2$ matrices contain $A_{31} \, A_{32}$ as the bottom row, and hence both of the given determinants (and their common divisors) would be divisible by the common divisor of these elements, proving $f_1 | f_2$.
	
	For the converse we note that given a matrix of determinant $1$, we can write 
	\[
		A_{31} = \left| \begin{matrix} M_{\{23\},\{13\}}& M_{\{23\},\{12\}}		 \\ M_{\{13\},\{13\}} & M_{\{13\},\{12\}} \end{matrix} \right|, 
		\qquad
		A_{32} = \left| \begin{matrix} M_{\{23\},\{23\}}& M_{\{23\},\{12\}}		 \\ M_{\{13\},\{23\}} & M_{\{13\},\{12\}} \end{matrix} \right|. 		
	\]
	
	This time the latter column of these two determinants are the same, and hence the common divisor of the $M_{\{13\},\{23\}}$ and  $M_{\{13\},\{12\}}$ divides both $A_{31}$ and $A_{32}$. This means that $f_2 | f_1$, giving us the desired equality.
\end{proof}

Noting that multiplying $A$ on the left by an element of $\Gamma_\infty$ corresponds to adding multiples of lower rows to higher rows, and on the right corresponds to adding multiples of leftward columns to rightward columns, we can see that $M_{\{12\},\{12\}}$ and $A_{33}$ are determined up to modulo $f$ and $A_{32}/f$ is determined up to modulo $d_1$, and $M_{\{13\}, \{12\}}$ are determined up to modulo $d_2$. In the next lemma we show that we can independently and simultaneously change to a different element in the arithmetic progressions by passing to a similar matrix.

\begin{lemma} \label{lem:arithmeticProgressionChoice}
Given an integral $3\times 3$ matrix $A$, we can find another integral matrix $A' = [A_{ij}']_{i,j}$ with  $A' \sim A$ such that for any $n_1, n_2,n_3,n_4 \in \Z$, 
\begin{center}
\begin{tabular}{ccccc}
	$A_{32}'/f$ & $=$ & $A_{32}/f $&$+$& $d_1 n_1$\\
	$A_{33}'$ &$=$ & $A_{33}$ &$+$& $fn_2$\\
	$M_{\{12\},\{12\}}'$ &$=$& $M_{\{12\},\{12\}}$  &$+$& $fn_3$\\
	$M_{\{13\}, \{12\}}'/f$ &$=$& $M_{\{13\}, \{12\}}/f$ &$+$& $d_2 n_4.$
\end{tabular}
\end{center}
\end{lemma}
	
\begin{proof}
	By Bezout's lemma we find $k$ and $\ell$ such that $k A_{31} + \ell A_{32} = \gcd(A_{32}, A_{31}) =f$. Thus we consider
	
\[
	 \left[\begin{matrix}
	A_{11}' & A_{12}' & A_{13}' \\ 
	A_{21}' & A_{22}' & A_{23}' \\
	A_{31}' & A_{32}' & A_{33}'
	\end{matrix}\right] = A\left[\begin{matrix} 1& n_1 & n_2k \\ 0&1& n_2 \ell\\ 0&0&1 \end{matrix} \right],
\]
which has the effect of adding $n_1$ times the first column of $A$ to the  second column of $A$, and $n_1k$ times the first column plus $n_1\ell$ times the second column to the third column of $A$. Therefore $A_{32}' = A_{32} + A_{31}n_2$ which implies $A_{32}'/f = A_{32}/f + d_1n_2$ and $A_{33}' = A_{33} + n_1 (kA_{31} + \ell A_{32})$.

Notice that $M_{\{12\},\{12\}}' = M_{\{12\},\{12\}}$ and $M_{\{13\},\{12\}}' = M_{\{13\},\{12\}}$. In order to change these entries we multiply by elements of $\Gamma_\infty$ on the left. Let us find $r,s \in \Z$ such that $rM_{\{23\},\{12\}}+sM_{\{13\},\{12\}} = f$, then we can obtain 
	\[
		A'' = [A_{ij}'']_{i,j} := \bpm 1&n_4&-n_3 r \\ 0&1&n_3 s\\ 0&0&1 \epm \bpm A_{11} & A_{12} & A_{13} \\ A_{21} & A_{22} & A_{23} \\ A_{31} & A_{32} & A_{33} \epm.
	\] 
	Here the minor $M_{\{12\},\{12\}}$ changes as,
	\begin{multline*}
		M_{\{12\},\{12\}}'' = \left| \begin{matrix}
		A_{11}'' & A_{12}''\\ 
		A_{21}'' & A_{22}''
		\end{matrix} \right| = \left| \begin{matrix}
		A_{11} & A_{12}\\ 
		A_{21} & A_{22}
		\end{matrix} \right| - n_3 r \left| \begin{matrix}
		A_{31} & A_{32}\\ 
		A_{21} & A_{22}
		\end{matrix} \right| + n_3s\left|\begin{matrix}
		A_{11} & A_{12}\\ 
		A_{31} & A_{32}
\end{matrix}\right| \\
 = M_{\{12\},\{12\}} + n_3( r M_{\{23\},\{12\}} + s M_{\{13\},\{12\}}) = M_{\{12\},\{12\}} + n_3 f.
	\end{multline*}
	A simpler calculation also shows $M_{\{13\},\{12\}}'' = M_{\{13\},\{12\}} + d_2f n_4$. This proves the desired equalities.
\end{proof}

\begin{lemma}\label{lem:x3y3}
	Given an integral matrix $A$ of determinant one, $A_{33} M_{\{12\},\{12\}} \equiv 1 \pmod f$, in particular $\gcd(A_{33},f) = \gcd( M_{\{12\},\{12\}}, f) = 1$. 
\end{lemma}
\begin{proof}
	Expanding the determinant of $A$ along the last row (or equivalently the last column) we get that
	\begin{multline*}
		1 = \det(A) = A_{31} M_{\{12\},\{23\}} - A_{32} M_{\{12\},\{13\}} +A_{33} M_{\{12\},\{12\}} \\
		= f(d_1 M_{\{12\},\{12\}} - (A_{32}/f) M_{\{12\},\{13\}}) +  A_{33} M_{\{12\},\{12\}}.\qedhere
	\end{multline*}
\end{proof}

Since the congruence subgroup $\Gamma_0(N)$ is defined by $N| A_{31}$, and $N| A_{32}$, which is equivalent to $N | \gcd(A_{31}, A_{32})$, 
we see that fine Kloosterman set decomposition is compatible with the level structure of $\Gamma_0(N)$. 

\begin{theorem}
	Let $\Gamma_0(N)\subseteq \SL_3(\Z)$ be the congruence subgroup such that the last row is congruent to $\bpm 0&0&* \epm \pmod{N}$. Then we have the decomposition
	\[
		\Gamma_\infty \backslash  (\Gamma_0(N) \cap Bw_0B)/ \Gamma_\infty = \bigsqcup_{\substack{f,d_1,d_2 \in \Z - \{0\} \\ N | f}} \Gamma_\infty \backslash \Omega(d_1,d_2,f) /\Gamma_\infty .
	\]
	Thus the only fine Kloosterman sums appearing in the $\SL_3$ Bruggeman-Kuznetsov trace formula are those with $N|f$. 
\end{theorem}

Thus we have, for example 
\[
	S_{w_0}^{\Gamma_0(2)}(\vec{m}, \vec{n}; 4,6 ) = \S_{w_0} (\vec{m},\vec{n}; 2,3,2) ,
\]	
which is not equal to $S_{w_0}(\vec{m}, \vec{n}; (c_1,c_2) )$ for any $c_1, c_2 \in \Z$. In short, the coarse Kloosterman sums for $\Gamma_0(N)$ can be given as finite sums of certain full level fine Kloosterman sums.


\subsection{Reduced Word Decomposition and the parametrization of the fine Kloosterman cells}

In the symmetric group $S_3$, let us call the simple transpositions $s_\alpha = (12) , s_\beta = (23)$. Using the reduced word decomposition $w_0 = s_\alpha s_\beta s_\alpha$ we parametrize the fine Kloosterman sets, that is, given 
\begin{equation}\label{eq:xymatrices}
	\gamma_2 = \bpm x_2 & b_2\\ d_2 & y_2 \epm,\quad \gamma_3 = \bpm x_3 & D \\ f & y_3 \epm,  \quad \gamma_1 = \bpm x_1 & b_1 \\ d_1 & y_1 \epm,
\end{equation}
we use the product 
\begin{equation}\label{eq:iproduct}
\iota_\alpha(\gamma_2) \iota_\beta(\gamma_3) \iota_\alpha(\gamma_1)
\end{equation}
to express elements of $\Omega(d_1,d_2,f)$.

Every product of the form \eqref{eq:iproduct} with the matrices \eqref{eq:xymatrices} in $\SL_2(\Z)$ gives an element of $\Omega(d_1,d_2,f)$. It is, however not true that any element of $\Omega(d_1,d_2,f)$ can be expressed as such a product. Firstly it is sometimes necessary to pick the matrices \eqref{eq:xymatrices} in $\SL_2(\Q)$, and secondly some matrices $A$ cannot be obtained in such a manner as can see by taking a matrix with $D =0$ in the notation of Proposition \ref{prop:BottSamelsonLikeDecomposition} below.
However it is possible to find a representative $A' \in \Gamma_\infty A \Gamma_\infty$ that factorizes.

Call $A_{33}M_{33} -1 = fD$.  By Lemma \ref{lem:x3y3} we see that $D$ is an integer. Using Lemma \ref{lem:arithmeticProgressionChoice} we can make sure that $D \neq 0$.

Specializing the result of Proposition \ref{prop:BruhatCoordinatesViaExterior} to our case, we have that for any matrix $A \in \SL_3$,
\[
	A = 	 \bpm 1 & \left. M_{23} \middle/ M_{13}\right. & A_{11} /A_{31} \\  & 1 & A_{21}/A_{31} \\ & &1 \epm w_0
		 \bpm A_{31} &&\\ &\frac{M_{13}}{A_{31}}&\\ &&\frac{1}{  M_{13}}  \epm 
		 \bpm 1 & A_{32}/A_{31} & A_{33}/A_{31} \\
		 & 1 & \left.M_{12} \middle/ M_{13} \right. \\ &&1\epm .
\]

\begin{proposition} \label{prop:BottSamelsonLikeDecomposition}
	Let $A$ be an integral matrix in the big Bruhat cell. Assume (by changing to a different element in the double coset $\U_3(\Z)A\U_3(\Z)$ if necessary) that $fD := A_{33}M_{33} - 1 \neq 0$. We have the explicit decomposition,
	\begin{multline*}
		A = \ea{ M_{23}/ f }{d_2} s_\alpha h_\alpha(d_2) \ea{A_{23}/D}{d_2} \eb{M_{33}}{f} s_\beta h_\beta(f) \\ 		 
	\times	 \eb{A_{33}}{f} \ea{M_{32}/D}{d_1} s_\alpha h_\alpha(d_1) \ea{\left.A_{32}\middle/ f \right.}{d_1}.
	\end{multline*}
\end{proposition}

This proposition states that the double cosets 
\[
\Gamma_\infty \iota_\alpha(\gamma_2) \iota_\beta(\gamma_3) \iota_\alpha(\gamma_1) \Gamma_\infty,
\]
with $\gamma_i$ as in \eqref{eq:xymatrices} with $b_i = \frac{x_i y_i -1}{d_i}$ for $i = 1,2$ and $b_3 = \frac{x_3 y_3 -1 }{f} = D$ and $x_2, y_1, x_3, y_3, b_3 \in \Z$ and $x_1, y_2 \in \frac{1}{D} \Z$ 
gives a surjective map onto $\Gamma_\infty\backslash \Omega(d_1,d_2,f)/ \Gamma_\infty$. Furthermore it is enough to take a single representative  $y_1 \pmod{d_1}, x_2 \pmod{d_2}$ and $x_3, y_3 \pmod{f}$.

\begin{proof}[Proof of Proposition \ref{prop:BottSamelsonLikeDecomposition}]
We begin with  Ansatz that $A$ is of the form $\iota_\alpha(\gamma_2) \iota_\beta(\gamma_3) \iota_\alpha(\gamma_1)$ with the coordinates of $\gamma_2, \gamma_3$ and $\gamma_1$ as in \eqref{eq:xymatrices}. Using the word-based factorization coordinates, we calculate
\begin{equation}\label{eq:xycoordinates}
\begin{split}
	A\vec{e}_{1,2} &= \iota_\alpha(\gamma_2) \iota_\beta(\gamma_3) \iota_\alpha(\gamma_1) \vec{e}_{1,2}= x_3 \vec{e}_{1,2} + x_2f \vec{e}_{1,3} + fd_2\vec{e}_{2,3},\\
A \vec{e}_{1,3} &= x_1 D \vec{e}_{1,2} + (x_1x_2 y_3 + d_1b_2)\vec{e}_{1,3} + (x_1 y_3 d_2 + d_1 y_2 )\vec{e}_{2,3},\\
A \vec{e}_{2,3} &= b_1D \vec{e}_{1,2} +  (b_1y_3x_2 + y_1b_2) \vec{e}_{1,3} + (b_2y_3 d_2 +  y_1y_2) \vec{e}_{2,3},\\
A\vec{e}_1 & = (x_1 x_2 + x_3  b_2 d_1)\vec{e}_1 + (x_1 d_2 + y_2 x_3 d_1) \vec{e}_2  + d_1f \vec{e}_3,\\
A \vec{e}_2 &= (b_1 x_2 + y_1 x_3 b_2)\vec{e}_1 + (b_1 d_2 + y_1 x_3 y_2) \vec{e}_2 + y_1 f\vec{e}_3,\\
A\vec{e}_3 &= Db_2 \vec{e}_1 + Dy_2 \vec{e}_2 + y_3 \vec{e}_3.
\end{split}
\end{equation}
Therefore one must have $A_{31} = d_1f$, $M_{13} = d_2f$, as well as, 
\begin{align*}
x_2 = \frac{\langle \vec{e}_{1,3}^* , A \vec{e}_{1,2} \rangle}{f} = \frac{\left| \begin{matrix} A_{11} & A_{12} \\ A_{31} & A_{32} \end{matrix}\right|}{f}, && 
y_2 = \frac{\langle \vec{e}_{2}^*, A\vec{e}_{3} \rangle}{D} = \frac{A_{23}}{D} ,
\\
x_3 = \langle \vec{e}_{1,2}^*, A \vec{e}_{1,2} \rangle = \left|\begin{matrix} A_{11} & A_{12} \\ A_{21} & A_{22}  \end{matrix}\right|, 
&& y_3 =  \langle \vec{e}_{3}^*, A\vec{e}_{3} \rangle = A_{33} ,
\\
x_1 = \frac{\langle \vec{e}_{1,2}^*, A\vec{e}_{1,3} \rangle}{D} = \frac{\left| \begin{matrix} A_{11} & A_{13}\\ A_{21} & A_{23} \end{matrix} \right|}{D},,
&& y_1 = \frac{\langle \vec{e}_{3}^*, A\vec{e}_{2} \rangle}{f} = \frac{A_{32}}{f}.
\end{align*}

Also from the fact that $f$ divides $A_{32}$ and $M_{23}$ we deduce that $x_2, x_3, y_3, y_1 \in \Z$ and $y_2,x_1 \in \tfrac{1}{D}\Z$. 

Multiplying these gets $A$ back, justifying our Ansatz. 
\end{proof}

Let $d_1, d_2, f$ be nonzero integers. Proposition \ref{prop:BottSamelsonLikeDecomposition} shows us that we can choose an $A$ from every $\U(\Z)$--double coset in $\Omega(d_1,d_2,f)$, such that  $A$ is of the form $\iota_\alpha(\gamma_2) \iota_\beta(\gamma_3) \iota_\alpha(\gamma_1)$ with the coordinates of $\gamma_2, \gamma_3$ and $\gamma_1$ as in \eqref{eq:xymatrices}. We can also assume $x_3$ to be any element in the arithmetic progression $A_{33} + f\Z$,  and similarly for $y_3 \in  M_{33} + f\Z$. Using Lemma \ref{lem:x3y3} both of these arithmetic progressions contain infinitely many primes, we can choose $x_3, y_3$ as primes larger than $d_1d_2f$. In particular we may assume that $(x_3, d_1d_2f) = (y_3, d_1d_2f) = 1$.

Also as $x_2, y_1$ can be chosen as arbitrary elements in $M_{12}/f + d_2\Z$ and $A_{32}/f + d_1 \Z$ respectively, we may also assume $x_2$ and $y_1$ to be relatively prime to $d_1d_2f$. 

Let us now express the coordinates of the Bruhat decomposition using these coordinates. So we write $A = u_L w_0 t u_R$ and also $A = \iota_\alpha(\gamma_2) \iota_\beta(\gamma_3) \iota_\alpha(\gamma_1)$. 

From Proposition \ref{prop:BruhatCoordinatesViaExterior} we know that
\[
	u_L = \bpm 1& \frac{\langle \vec{e}^*_{1,3},A \vec{e}_{1,2}\rangle}{t_1t_2} & \frac{\langle \vec{e}^*_{1}, A\vec{e}_1 \rangle }{t_1}\\ &1&\frac{\langle \vec{e}^*_2, A\vec{e}_1\rangle }{t_1} \\ &&1 \epm \qquad \text{ and } \qquad u_R = \bpm 1& \frac{\langle \vec{e}_3^*, A \vec{e}_2 \rangle }{t_1} & \frac{\langle \vec{e}^*_3, A\vec{e}_3\rangle }{t_1}\\ & 1& \frac{\langle \vec{e}^*_{2,3}, A\vec{e}_{1,3} \rangle }{t_1t_2} \\ &&1 \epm.
\]
Denoting $u = x_1d_2 + y_2 x_3 d_1$, and $v = x_1 y_3 d_2 + y_2 d_1$, we look at \eqref{eq:xycoordinates} and deduce
\begin{align*}
	&\langle \vec{e}^*_{1,3}, A \vec{e}_{1,2} \rangle  = x_2 f,   
	&\langle \vec{e}^*_{2,3},A\vec{e}_{1,3} \rangle  = x_1 y_3 d_2 + d_1 y_2 = v, \\	
	&\langle \vec{e}^*_{3}, A\vec{e}_2 \rangle = y_1 f ,
	&\langle \vec{e}_{2}^*, A\vec{e}_1 \rangle =   x_1 d_2 + y_2 x_3 d_1 = u ,
\end{align*}
and
\begin{equation*}
	\langle \vec{e}^*_{1}, A\vec{e}_1 \rangle = (x_1 x_2 + x_3  b_2 d_1)  = \frac{x_2 u - x_3 d_1}{d_2}, \qquad \qquad \langle \vec{e}^*_{3}, A\vec{e}_3 \rangle   = y_3.
\end{equation*}
Notice that $u,v\in \Z$. Combining these calculations, we obtain the following result.

\begin{proposition}
	Given a matrix $A \in \SL_3(\Z)$, choose $d_1,d_2,f$ as in \eqref{eq:d1d2fDefinition}. After replacing $A$ with $A' \sim A$ if necessary, we can write $A =  \iota_\alpha(\gamma_2) \iota_\beta(\gamma_3) \iota_\alpha(\gamma_1)$ with $\gamma_i$ as in \eqref{eq:xymatrices}, and $u,v\in \Z$ as above its Bruhat decomposition has the coordinates
	\begin{equation} \label{eq:xyuvCoordinates}
		A = \bpm 1&\frac{x_2}{d_2} & \frac{x_2 u - x_3 d_1}{d_1d_2f}\\ & 1& \frac{u}{d_1f}\\ &&1 \epm \bpm &&1\\&-1&\\ 1&&\epm \bpm d_1f &&\\ &\frac{d_2}{d_1} & \\ &&\frac{1}{d_2f} \epm \bpm 1& \frac{y_1}{d_1} & \frac{y_3}{d_1f} \\
		&1& \frac{v}{d_2 f} \\ &&1 \epm,
	\end{equation}
	with all the visible parameters integral, $x_2,y_1,x_3,y_3$ relatively prime to $d_1d_2f$, and $x_3 y_3 \equiv 1  \pmod f$.
\end{proposition}

In the next proposition we give the conditions under which the coordinates in \eqref{eq:xyuvCoordinates} give rise to the same $\Gamma_\infty$--double coset.

\begin{proposition} \label{prop:congruenceEquations}
	Given nonzero integers $d_1, d_2, f$ and $y_1 \in (\Z/ d_1\Z)^*$, $x_2 \in (\Z/d_2\Z)^*$ and $x_3, y_3 \in \Z/f\Z$ satisfying $x_3y_3 \equiv 1 \pmod f$, the product in \eqref{eq:xyuvCoordinates} gives rise to an integral matrix if and only if the following congruence conditions are satisfied:
	\begin{align}
    ux_2 &\equiv d_1 x_3 \pmod{d_2},
\label{eq:firstcongequation}\\ 
    uy_1 &\equiv d_2 \pmod{d_1},\\
    ux_2y_1 &\equiv d_1x_3 y_1 + d_2 x_2 \pmod{d_1d_2}, \label{eq:ucongequation}\\
    v &\equiv uy_3 \pmod{d_1f},\\
    vx_2 &\equiv u y_3 x_2 + d_1(1 - x_3y_3) \pmod{d_1d_2f}. \label{eq:vcongequation}
	\end{align}
	Furthermore a matrix $B$ that formed in the same way from the coordinates $Y_1, X_2, U, V$ and $x_3, y_3$ is in $\Gamma_\infty A \Gamma_\infty$ if and only if 
	\begin{align*}
		y_1 &\equiv Y_1 \pmod{d_1},\\
		x_2 &\equiv X_2 \pmod{d_2},\\
		u &\equiv U \pmod{d_1d_2f},\\
		v & \equiv V \pmod{d_1d_2f}.
	\end{align*}
\end{proposition}
\begin{remark}
Notice that we have forced an equality in the $x_3, y_3$ coordinates. This is because $A\sim B$ for $A, B \in \Omega(d_1,d_2,f)$ implies that $\langle \vec{e}_{1,2}^*, A\vec{e}_{1,2} \rangle \equiv \langle \vec{e}_{1,2}^*,B\vec{e}_{1,2} \rangle \pmod{f}$ and $\langle \vec{e}_{3}^*, A\vec{e}_{3} \rangle  \equiv  \langle \vec{e}_{3}^*, B\vec{e}_{3} \rangle  \pmod{f}$. Furthermore if these congruences are satisfied equality can be achieved via Lemma \ref{lem:arithmeticProgressionChoice}, by switching to a representative $B' \sim B$. Thus we can force the equalities $x_3 = X_3$ and $y_3 = Y_3$. 
\end{remark}

\begin{remark}
	If we choose $y_1, x_2$ to be relatively prime to $d_1d_2f$ (which we can) then \eqref{eq:ucongequation} and \eqref{eq:vcongequation} imply the remaining congruence relations.
\end{remark}

\begin{proof}[Proof of Proposition \ref{prop:congruenceEquations}]
Multiply the product in \eqref{eq:xyuvCoordinates}, obtaining
\[
	A = \left(\begin{matrix}
\frac{u x_{2} - d_{1} x_{3}}{d_{2}} & \frac{u x_{2}y_1 - d_{1} x_{3} y_{1} - x_2 d_2}{d_{1} d_{2}} & \frac{-v x_{2} + ux_2 y_3 + d_1 (1 - x_3 y_3)}{d_{1} d_{2} f}  \\
u & \frac{u y_{1} - d_2}{d_{1}}  & \frac{u y_{3} - v}{d_{1} f}  \\
d_{1} f & f y_{1} & y_{3}
\end{matrix}\right).
\]
This is an integral matrix if and only if the congruences \eqref{eq:firstcongequation}--\eqref{eq:vcongequation} are satisfied. 

Two matrices $A = u_L(A)w_0tu_R(A)$ and $B = u_L(B)w_0tu_R(B)$ are in the same $\Gamma_\infty$--double coset if and only if $u_L(A)u_L(B)\inv, u_R(B)\inv u_R(A) \in \Gamma_\infty.$
Denoting the coordinates of $B$ via capital letters, $u_L(A) u_L(B)\inv\in \Gamma_\infty$ means 
\[
\bpm 1&\frac{x_2}{d_2} & \frac{x_2u-d_1x_3}{d_1d_2f}\\ &1& \frac{u}{d_1f} \\ &&1 \epm \bpm 1& -\frac{X_2}{d_2} & \frac{d_1X_3}{d_1d_2 f}\\ &1&-\frac{U}{d_1 f} \\ &&1 \epm = \bpm 1& \frac{x_2 - X_2}{d_2} & \frac{x_2 (u - U) + d_1(X_3-x_3) }{d_1d_2 f} \\ & 1& \frac{u - U}{d_1 f} \\ &&1\epm\in \Gamma_\infty.
\]
This forces $x_2 \equiv X_2 \pmod{d_2}$ and $u \equiv U \pmod{d_1f}$. Now taking $x_3 = X_3$ the upper right corner forces us to have $d_1d_2f | x_2(u - U)$. Notice that we already have $u - U$ divisible by $d_1f$ and $x_2$ is relatively prime to $d_2$. Therefore the above matrix is integral if and only if
\[
	x_2 \equiv X_2 \pmod{d_2} \qquad u \equiv U \pmod{d_1d_2f}. 
\]

Similarly $u_R(B)\inv u_R(A) \in \Gamma_\infty$, i.e.,
\[
	 \bpm 1& -\frac{Y_1}{d_1} & \frac{-Y_3d_2 +Y_1V}{d_1d_2f} \\ & 1 & -\frac{V}{d_2f} \\ &&1\epm \bpm 1& \frac{y_1}{d_1} &\frac{y_3}{d_1f}\\ &1&\frac{v}{d_2f} \\ & & 1 \epm = \bpm 1&\frac{y_1 - Y_1}{d_1} & \frac{(y_3 - Y_3)d_2 + Y_1(V - v)}{d_1d_2f } \\ &1& \frac{v-V}{d_2f}\\ &&1\epm \in \Gamma_\infty
\]
implies (after taking $y_3 = Y_3$) that
\[
	y_1 \equiv Y_1 \pmod{d_1} \qquad v \equiv V \pmod{d_1d_2f}.\qedhere
\]
\end{proof}

From now on we will assume $x_2$ and $y_1$ are chosen to be relatively prime to $d_1d_2f$. 

Since the equation \eqref{eq:ucongequation} determines $u$ up to $d_1d_2$ but $u$ determines the double coset up to modulo $d_1d_2f$, the set of allowed solutions are
\begin{equation} \label{eq:uparametrization}
	u \equiv d_1 x_3 \overline{x_2} + d_2 \overline{y_1} + d_1d_2 k \pmod{d_1d_2 f},
\end{equation}
with $k \in \Z / f\Z$. 

This then determines $v \pmod{d_1d_2f}$ completely and we have for each such $u$, 
\[
	v \equiv (	d_1 x_3 \overline{x_2} + d_2 \overline{y_1} + d_1 d_2k  ) y_3 + d_1(1 - x_3 y_3)\overline{x_2} \equiv  d_2 \overline{y_1} y_3 + d_1 \overline{x_2} + d_1 d_2 y_3 k  \pmod{d_1d_2 f}.
\]

This gives a parametrization of the fine Kloosterman cells.

\begin{corollary} \label{cor:FineKloostermanCellBijection}
	Let $d_1, d_2, f$ be nonzero integers, and fix the sets $\mathcal{Y}_{d_1}$ and $\mathcal{X}_{d_1}$, a complete set of reduced residue class representatives $y_1 \pmod{d_1}^*, x_2 \pmod{d_2}^*$ such that $x_2, y_1,$ are relatively prime to $d_1d_2f$. Let $\mathcal{F}_f = \{(x_3,y_3) \in \{f+1, \ldots, 2f\} | x_3 y_3  \equiv 1 \pmod{f}\}$ and let $k \in \mathcal{K}_f$ simply run through integers from $0$ to $f-1$. There is a bijection 
\[
\begin{array}{ccc}
		\mathcal{X}_{d_2} \times \mathcal{Y}_{d_1} \times \mathcal{F}_f \times \mathcal{K}_F & \longrightarrow & \Gamma_\infty \backslash \Omega(d_1,d_2,f) / \Gamma_\infty \\
		&&\\
		(x_2, y_1, (x_3, y_3), k) & \longmapsto & \Gamma_\infty \left(\begin{matrix}
\frac{u x_{2} - d_{1} x_{3}}{d_{2}} & \frac{u x_{2}y_1 - d_{1} x_{3} y_{1} - x_2 d_2}{d_{1} d_{2}} & \frac{-v x_{2} + ux_2 y_3 + d_1 (1 - x_3 y_3)}{d_{1} d_{2} f}  \\
u & \frac{u y_{1} - d_2}{d_{1}}  & \frac{u y_{3} - v}{d_{1} f}  \\
d_{1} f & f y_{1} & y_{3}
\end{matrix}\right)\Gamma_\infty,
\end{array}
\]
	where
	$u = d_1 x_3 \overline{x_2} + d_2 \overline{y_1} + d_1d_2 k$ and $v= d_2 \overline{y_1} y_3 + d_1 \overline{x_2} + d_1 d_2 y_3 k$.
\end{corollary}

\begin{remark}
	The condition that $f < x_3, y_3< 2f$ is not important. Any fixed set of reduced residue classes would work as long as $x_3 y_3 -1\neq 0$.
\end{remark}

\begin{corollary}
	The number of elements in the coarse Kloosterman set double coset $\Omega_{w_0}( c_1,c_2)$ is given by
	\[
		\left|\Gamma_\infty \backslash \Omega_{w_0}( c_1,c_2)/ \Gamma_\infty \right| = \sum_{f | (c_1, c_2)}  \phi\left( \frac{c_1}{f} \right) \phi\left(\frac{c_2}{f}\right) \phi(f) f.
	\]
\end{corollary}

\subsection{Evaluation of Fine Kloosterman Sums} \label{subsec:FineKloostermanEvaluation}

According to this parametrization we evaluate $\S_{w_0}(\vec{m}, \vec{n}; d_1,d_2,f)$. The $k$ sum will give us a restriction on the set of $(x_3, y_3)$ pairs as well as the condition that $(n_2d_1, f) = (m_2d_2,f)$.

\begin{theorem}\label{thm:KloostermanEvaluation}
Let $n_1, n_2, m_1,m_2 \in \Z$. The Kloosterman sum $\S_{w_0}(\vec{m}, \vec{n}; d_1,d_2,f)$ is zero unless $(m_2d_2,f) = (n_2 d_1, f)$. If this is satisfied, then the Kloosterman sum equals,
\[
	f \sum_{\substack{x_3,y_3 \pmod{f}\\ x_3 y_3 \equiv 1 \pmod{f}\\ m_2d_2 + n_2d_1 y_3 \equiv 0 \pmod{f} }}S(n_1, (m_2d_2 + n_2d_1y_3)/f; d_1)  S(m_1,(n_2d_1 + m_2d_2x_3)/f; d_2).
\]
\end{theorem}
\begin{proof}
We calculate by using the definition of the fine Kloosterman sum, the coordinatization of the Kloosterman set from \ref{cor:FineKloostermanCellBijection}, and the explicit form of the superdiagonal elements in the unipotent factors of the Bruhat decomposition in terms of these coordinates as in \eqref{eq:xyuvCoordinates},
\begin{multline*}
	\S_{w_0}(\vec{m}, \vec{n}; d_1,d_2,f) = \sum_{\substack{\gamma \in \Gamma_\infty \backslash \Omega(d_1,d_2,f)/\Gamma_\infty  \\ \gamma \in u_L w_0t(d_1f, d_2f ) u_R} } \psi_{(m_1,m_2) } (u_L) \psi_{(n_1,n_2)}(u_R)\\
	= \sum_{\substack{x_2 \in \mathcal{X}_{d_2} \\ y_1 \in \mathcal{Y}_{d_1}}} \sum_{(x_3, y_3) \in \mathcal{F}_f} \sum_{k = 0}^{f-1} e\left(\frac{m_1x_2}{d_2} + \frac{m_2u}{d_1f} + \frac{n_1y_1}{d_1} + \frac{n_2v}{d_2f}\right).
\end{multline*}
Then we plug in the values for $u$ and $v$ in terms of the given coordinates,
\begin{multline*}
		\S_{w_0}(\vec{m}, \vec{n}; d_1,d_2,f) = \sum_{\substack{x_2 \in \mathcal{X}_{d_2} \\ y_1 \in \mathcal{Y}_{d_1}}} \sum_{(x_3, y_3) \in \mathcal{F}_f} e\left(\frac{m_1x_2}{d_2} +  \frac{n_2d_1\overline{x_2}}{d_2f}  + \frac{m_2 x_3 \overline{x_2}}{f} \right)\\
		\times e\left(\frac{n_1 y_1}{d_1} + \frac{m_2d_2\overline{y_1}}{d_1f} +  \frac{n_2 y_3 \overline{y_1}}{f} \right) \sum_{k= 0}^{f-1} e\left(\frac{m_2 d_2 + n_2d_1 y_3}{f}k \right).
\end{multline*}
The innermost sum over $k$ gives us the congruence condition
\begin{equation}\label{eq:y3condition}
	m_2d_2 + n_2d_1 y_3 \equiv 0 \pmod{f},
\end{equation}
for otherwise the sum vanishes. This is only satisfiable for some $y_3 \in (\Z/f\Z)^*$ if $(m_2d_2, f) = (n_2d_1,f)$. Thus, 
\begin{multline*}
		\S_{w_0}(\vec{m}, \vec{n}; d_1,d_2,f) = f \sum_{\substack{(x_3, y_3) \in \mathcal{F}_f\\ m_2d_2 + n_2d_1y_3 \equiv 0\pmod{f}}} 
		 \sum_{y_1 \in \mathcal{Y}_{d_1}} e\left(\frac{n_1 y_1}{d_1} + \frac{(m_2d_2+ n_2d_1y_3) \overline{y_1}}{d_1f}  \right) \\
		 \times \sum_{x_2 \in \mathcal{X}_{d_2}} e\left(\frac{m_1 x_2}{d_2} + \frac{(m_2d_2x_3 + n_2d_1)\overline{x_2}}{d_2f}\right).
\end{multline*}

Let $y_3$ be chosen so that \eqref{eq:y3condition} is satisfied. Define the integers $N = N(y_3) := (m_2d_2 + n_2d_1 y_3)/f$ and $M  = M(x_3) := (m_2d_2x_3 + n_2 d_1)/f$. These are both integers, due to the condition on $(x_3, y_3)$. Note that if $x_3 \equiv x_3' \pmod{f}$ then $M(x_3) \equiv M(x_3') \pmod{d_1}$ and similarly for $N(y_3)$. The fine Kloosterman sum is
 \[
  \S_{w_0} (\vec{m}, \vec{n}; d_1,d_2,f) = f \sum_{\substack{x_3, y_3 \pmod{f}\\x_3y_3 \equiv 1 \pmod{f}\\ m_2d_2 + n_2d_1y_3 \equiv 0 \pmod{f}}}S(n_1,N(y_3);d_1) S(m_1,M(x_3);d_2).
 \]
\end{proof}

We highlight a special case for the Kloosterman sum. Assume that $(f,d_1d_2) = 1$. For example this is the only kind of fine Kloosterman sum that appears in the stratification of the coarse Kloosterman sum $S_{w_0}(\vec{m}, \vec{n}; (c_1,c_2))$ if $c_1c_2$ is cube-free.

\begin{proposition}\label{prop:coprimeSimplification}
	Assume $d_1, d_2$ and $f$ are such that $f$ is relatively prime to $d_1d_2$. Let $n_1, n_2, m_1,m_2 \in \Z$ such that $(n_2, f) = (m_2,f) = e$, and $f = eh$. Let $h^*$ be the $h$-primary part of $f$, so that $h^*|f$ contains the same prime factors as $h$ and $f/h^*$ is relatively prime to $h$. Let us choose $\overline{f}$ so that $\overline{f} f \equiv 1 \pmod{d_1d_2}$. Then,
\[
	\S_{w_0}(\vec{m}, \vec{n};d_1,d_2,f)= f \phi\left(\frac{f}{h^*}\right) \frac{h^*}{h} S(n_1, m_2d_2\overline{f}; d_1) S(m_1,n_2d_1\overline{f}; d_2) .
\]

Using this we can have, using the Weil Bound,
\[
	|\S_{w_0}(\vec{m},\vec{n};d_1,d_2,f)| \leq f e  (m_1,n_2d_1,d_2)^{\frac12}d_2^{\frac{1}{2}} (n_1,m_2d_2,d_1)^{\frac12} d_1^{\frac12} \tau(d_1)\tau(d_2),
\]
where $\tau(d)$ is the divisor function.
\end{proposition}
\begin{proof}
	This follows from the previous proposition. Since $d_1$ and $d_2$ are relatively prime to $d_1$ and $d_2$ and since for the classical Kloosterman sum $S(n,m;d)$ the values of $n$ and $m$ are only important modulo $d$, we simply calculate,
	\[
		\frac{(m_2d_2 + n_2d_1y_3)}{f} \equiv \frac{(m_2d_2 + n_2d_1y_3)f \overline{f}}{f} \equiv m_2d_2\overline{f} + n_2d_1 y_3 \overline{f} \equiv  m_2 d_2 \overline{f} \pmod{d_1},
	\]  
	and similarly $(m_2d_2x_3 + n_2d_1)/f \equiv n_2d_1 \overline{f} \pmod{d_2}$.

	 Thus
\[
\S_{w_0} (\vec{m}, \vec{n}; d_1,d_2,f) = f \sum_{\substack{x_3, y_3 \pmod{f}\\x_3y_3 \equiv 1 \pmod{f}\\ m_2d_2 + n_2d_1y_3 \equiv 0 \pmod{f}}} S(n_1, m_2d_2 \overline{f}; d_1) S(m_1, n_2d_1 \overline{f}; d_2).
\]
	The summands have no $(x_3,y_3)$ dependence. However depending on the common factors of $n_2$ and $m_2$ with $f$ we have different number of solutions to $m_2d_2 + n_2d_1y_3 \equiv 0 \pmod{f}$. There is one solution modulo $f/e = h$, obtained by cancelling the common factors. Then each of these solutions lift to $h^*/h$ many solutions modulo $h^*$. Preserving the fact that the solutions need to be relatively prime to $f$, each of these solutions lift to $\phi\left(\frac{f}{h^*}\right)$ many solutions.
\end{proof}

As an exercise we can explicitly calculate \emph{coarse} Kloosterman sums. We calculate, for an odd prime $p$, 
\[
	S_{w_0}((1,p), (1,p); (p^2,p)) = \S_{w_0} ((1,p), (1,p) ; p^2, p,1) + \S_{w_0}((1,p), (1,p); p,1,p).
\]
The first term with $f = 1$ is easy to calculate, we can take $x_3 = y_3 = 0$ in \eqref{eq:FineKloostermanEvaluation} and get,
\[
	\S_{w_0}((1,p), (1,p); p^2,p,1) = S(1,p^2; p)S(1,p^3;p) = \mu(p^2) \mu(p) = 0.
\]
The second fine Kloosterman sum can be evaluated as
\[
	\S_{w_0}((1,p), (1,p); p,1,p) = p \hspace{-0.7cm} \sum_{\substack{x_3 y_3 \equiv 1 \pmod{p}\\ p + y_3 p^2 \equiv 0 \pmod{p}} }  \hspace{-0.7cm} S(1, \tfrac{p + p^2 y_3}{p}; p)S(1,\tfrac{p x_3 + p^2 }{p};1) = p (p-1) S(1,1;p), 
\]
since $(p-1)$ many $(x_3, y_3)$ pairs all yield the same answer. Thus we get  
\begin{equation} \label{eq:KloostermanSumEvaluationExample} 
\S_{w_0}((1,p), (1,p); (p^2,p)) = p(p-1) S(1,1;p).
\end{equation}

Another example would be $S_{w_0}((1,1),(p,p),(p,p)) = 2-p$. 

Finally let's take integers $m_1,m_2,n_1,n_2$ all coprime to $p$.  
\[
	S_{w_0}(\vec{m}, \vec{n};(p,p)) = \S_{w_0}(\vec{m},\vec{n};p,p,1) + \S_{w_0}(\vec{m}, \vec{n}; 1,1,p).
\]
The $f = 1$ case is simply $S(n_1,m_2p;p) S(m_1,n_2p;p) = c_{p}(n_1) c_{p}(m_1) = \mu(p)^2 = 1$ and the $f = p$ case is $p S(n_1,(m_2 + n_2 y_3)/p;1)S(m_1, (m_2x_3 + n_2)/p;1)$ for the unique $(x_3,y_3)$ pair modulo $p$, that makes the second arguments integers. Thus we get $p$. 
Together we get the identity \cite[(1.3)]{BlomerButtcaneGlobalDecomposition}, i.e. that $S(\vec{m}, \vec{n};(p,p)) = p+1$.

\subsection{The Braid Relation}
\label{subsec:BraidRelation}

The stratification $\Omega(d_1,d_2,f)$ was based on the long word decomposition, $w_0 = s_\alpha s_\beta s_\alpha$. 

Let now use the reduced word decomposition $w_0 = s_\beta s_\alpha s_\beta$. Within this subsection we call $\Omega(d_1,d_2,f) = \Omega_{(\alpha, \beta, \alpha)}(d_1,d_2,f)$ so that we can distinguish from $\Omega_{(\beta, \alpha, \beta)}(d_1,d_2,f)$. 

\begin{proposition} \label{prop:braidRelation}
	Given $A \in \SL_3(\Z) \cap Bw_0B$ as above, we have $$\gcd(A_{31}, A_{21} )= \gcd\left(\left|\begin{matrix} A_{21}& A_{22} \\ A_{31} & A_{32} \end{matrix} \right|, \left|\begin{matrix} A_{21}& A_{23} \\ A_{31} & A_{33} \end{matrix}\right|\right).$$
	Call this common $\gcd$ as $f$. Define
	\[
		\Omega_{(\beta,\alpha,\beta)}(d_1, d_2, f):= \left\{A \in \SL_3(\Z)| \gcd(A_{31}, A_{21}) = f, A_{31} = d_1f, M_{13} = d_2 f\right\}.
	\]
	These sets are right and left $\Gamma_\infty$-invariant. For every $A \in \Omega_{(\beta,\alpha,\beta)}$ by changing to another $A' \sim A$ if necessary, we can make sure that $D:= A_{11} M_{11} -1 \neq 0$. For such a matrix we can write it in the form $A = \iota_\beta(\gamma_1) \iota_\alpha(\gamma_3) \iota_\beta(\gamma_2)$ as follows
	\begin{multline*}
		A = e_\beta\left(\frac{A_{21}/f}{d_1}\right) s_\beta h_\beta(d_1) e_\beta\left(\frac{M_{21}/D}{d_1}\right)  \times e_\alpha \left(\frac{A_{11}}{f}\right) s_\alpha h_\alpha(f) e_\alpha\left(\frac{M_{11}}{f}\right) \\
		\times e_\beta\left( \frac{A_{12}/D}{d_2} \right) s_\beta h_\beta(d_2) e_\beta\left(\frac{M_{12}/f}{d_2}\right).
	\end{multline*}
	Using the coordinate names of \eqref{eq:xymatrices}, and calling $u = x_2 d_1 + x_3 d_2 y_1$ and $v = x_2 d_1 y_3 + d_2 y_1$, we can parametrize all the double cosets in $\Omega_{(\beta, \alpha, \beta)}(d_1, d_2,f)$:
	\[
		A = \bpm 1 & \frac{u}{d_2f}  & \frac{x_3}{d_1f} \\ & 1 & \frac{x_1}{d_1} \\ &&1  \epm w_0 \bpm d_1f && \\ &\frac{d_2}{d_1}& \\ &&\frac{1}{d_2f} \epm \bpm 1 &  \frac{v}{d_1f} & \frac{y_2v - y_3d_1}{d_1d_2f} \\ & 1& \frac{y_2}{d_2} \\ &&1 \epm .
	\]
	These coordinates will give rise to an $A\in Bw_0B \cap \SL_3(\Z)$ if and only if $x_3y_3 \equiv 1 \pmod{f}$, $x_1 v \equiv d_2 \pmod{d_1}$, $y_2v \equiv y_3 d_1 \pmod{d_2}$, $x_3 v \equiv u \pmod{d_1f}$, $x_1y_2 v \equiv y_3 x_1 d_1 + y_2 d_2 \pmod{d_1d_2}$, and $ \quad x_2 y_2 v + d_1(1 - x_3 y_3)  \equiv u y_2 \pmod{d_1d_2f} $. 
		
	Another matrix $B$ with the coordinates $X_1, Y_2, x_3, y_3, U, V$ is in the same $\Gamma_\infty$--double-coset if and only if
	\[
		u \equiv U \pmod{d_1d_2f},  \quad x_1 \equiv X_1 \pmod{d_1} \quad \text{ and } y_2 \equiv Y_2 \pmod{d_2}. 
	\]
	
	Using these coordinates the Kloosterman sum $\S_{(\beta,\alpha,\beta)}(\vec{m}, \vec{n};d_1,d_2,f)$ defined by restricting the summation in \eqref{eq:CoarseKloostermanDefinition} to the set of double cosets, $\Omega_{(\beta,\alpha,\beta)}(d_1,d_2,f)$ is given by
	\begin{multline*}
		\S_{(\beta,\alpha,\beta)}(\vec{m}, \vec{n};d_1,d_2,f) = \sum_{\substack{\gamma \in \Gamma_\infty \backslash \Omega_{(\beta, \alpha,\beta)} (d_1,d_2,f)/ \Gamma_\infty \\ \gamma = u_L w_0 t(d_1f, d_2f) u_R  }} \psi_{\vec{m}} (u_L) \psi_{\vec{n}} (u_R)\\
		= f \sum_{\substack{x_3, y_3\pmod f\\ x_3y_3 \equiv 1 \pmod{f}\\ n_1d_2 + x_3m_1d_1 \equiv 0\pmod{f}}} S((n_1d_2 + x_3 m_1 d_1)/f, m_2; d_1) S((n_1 d_2 y_3 + m_1d_1)/f, n_2; d_2).
	\end{multline*}
\end{proposition}

We leave the proof of this proposition to the reader as an exercise, as it is identical to the proofs in the case of the reduced word decomposition $w_0 = s_\alpha s_\beta s_\alpha$, \emph{mutatis mutandis}. 

Note that, given $A = \iota_\alpha(\gamma_2) \iota_\beta(\gamma_3) \iota_\alpha(\gamma_1)$ the involution
\[
	 A^\dagger :=  w_0 (\prescript{t}{}{A}\inv) w_0\inv
\]
is a homomorphism fixing $\U(\Z)$, and therefore it preserves $\Gamma_\infty$--double cosets. This involution does not preserve our finer decomposition, nor the set $\Omega_{w_0}(c_1,c_2)$ in general. However it sends the stratification based on one reduced word decomposition to the other. Indeed 
\[
	A^\dagger = \iota_\beta (\gamma_2) \iota_\alpha(\gamma_3) \iota_\beta(\gamma_1),
\]
and therefore we have the isomorphism
\begin{align*}
	\Omega_{(\alpha, \beta, \alpha)}(d_1,d_2,f) &\longrightarrow \Omega_{(\beta, \alpha, \beta)} (d_2,d_1,f)\\
	A &\longmapsto A^\dagger	.
\end{align*}
Notice the $d_1 \leftrightarrow d_2$ switch. This swapping can also be observed by noticing that the entries of $A^\dagger$ are given by $A^\dagger = \bsm M_{33} & M_{32} & M_{31}\\ M_{23} & M_{22} & M_{21} \\ M_{13} & M_{12} & M_{11} \esm$.

At this point we are inclined to think about the intersection of the strata  coming from the two reduced word decompositions, to obtain an even finer decomposition. That would correspond to controlling for the greatest common divisor of $A_{31}$ with both $A_{32}$ and $A_{21}$. Let $d_1,d_2,f_1,f_2,e$ be nonzero integers with $\gcd(f_1,f_2) = 1$. We define $\Omega_{w_0}^! = \Omega_{w_0}^! (d_1,d_2,f_1,f_2,e)$ as,
\[
	\Omega_{w_0}^{!} := \left\{ A \in BwB \cap \SL_3(\Z): \begin{matrix} \gcd(A_{31}, A_{32}) = f_1e, \,\, \gcd(A_{31}, A_{21}) = f_2e,\\ A_{31} = d_1f_1f_2e, \,\,\left| \begin{matrix} A_{21} &A_{22} \\ A_{31} & A_{32} \end{matrix} \right| = d_2 f_1 f_2 e \end{matrix} \right\}.
\] 
Notice that this is exactly the intersection of the two stratifications one obtains from the two reduced word decompositions. To be precise, if $\gcd(f_1,f_2) = 1$,
\[
	\Omega_{(\alpha,\beta,\alpha)} (d_1f_2, d_2f_2, f_1e) \cap \Omega_{(\beta, \alpha, \beta)} (d_1f_1, d_2f_1, f_2e) = \Omega_{w_0}^! (d_1,d_2, f_1,f_2,e).
\]
If the two strata corresponding to the two reduced word decompositions were in any other form, then their intersection would be empty.

One can also think about the Kloosterman sums restricted to these sets. We may define,
\[
	\S_{w_0}^!(\vec{m}, \vec{n}; d_1,d_2,f_1,f_2,e) := \sum_{\substack{\gamma \in \Gamma_\infty \backslash \Omega_{w_0}^!(d_1,d_2,f_1,f_2,e)/ \Gamma_\infty\\ \gamma = u_L w_0 t(d_1f_1f_2 e, d_2 f_1 f_2 e) u_R}} \psi_{\vec{m}}(u_L) \psi_{\vec{n}}(u_R).
\]
There is a good motivation to give a beautiful and comprehensible expression for these Kloosterman sums as we have done in Theorem \ref{thm:KloostermanEvaluation} for $\S_{w_0}(\vec{m}, \vec{n}; d_1,d_2,f)$. This is because the geometric side of the Kuznetsov-Bruggeman trace formula for the congruence subgroup
\[
	\Gamma_{00}(N) :=\left \{A \in \SL_3(\Z) : A \equiv \bsm * &* &* \\ 0&*&*\\ 0&0&* \esm \pmod{N} \right\}
\]
would be given by $\S_{w_0}^!(\vec{m}, \vec{n}, d_1,d_2,f_1,f_2,e)$ with only the condition $N|e$. 

If we are to follow the maxim that a good algebraic structure will lead to a beautiful comprehensible formula, a canonical set of double cosets that are induced from all reduced word decompositions is promising. However we were not able to obtain an aesthetically pleasing formula for this yet finer decomposition. This could be our own shortcoming, on the other hand not having such a formula perhaps understandable, a finer subdivision of $\Gamma_\infty$--double-cosets of $\Omega_{w_0}(c_1, c_2)$ should not automatically mean that the sum of $\psi_{\vec{m}}(u_L) \psi_{\vec{n}}(u_R)$ over these double cosets will comprise a comprehensible unit. Indeed, if we were to subdivide the set into singletons we would have a single exponential term, and have no hope for making use of cancellation.

\section{Ramanujan Sums and Triple Divisor Functions} \label{sec:RamanujanFormula}

One way to prove Ramanujan's formula \eqref{eq:ramanujanFormula} for the divisor sum is to calculate the $0$\th\ Fourier coefficient of an Eisenstein series $E(z,s) = \sum_{m,n\in \Z, (m,n)=1} |m + nz|^{-s}$, in two different ways.

This formula is an essential step in many of the $L$-function moment calculations such as \cite{YoungFourthMoment, Motohashi93}, to name a few. It allows us to display the \emph{guts} of the divisor function that appear in the moment calculations. After that we can, armed with transforms as scalpels, be very precise with our calculations and obtain sharp results.

Bump in \cite{BumpGL3} has given the extension of this formula to $\SL_3$. Using our formula for the Kloosterman sum, we can make the generalized Ramanujan sums more explicit. We first generalize the divisor sum, taken from Bump, ibid.
Let $\sigma_{\nu_1, \nu_2} (n_1,n_2)$ be defined multiplicatively, i.e. if $\gcd(n_1n_2, n_1'n_2') = 1$ then let
\begin{equation}\label{eq:sigmaMultiplicative}
	\sigma_{\nu_1, \nu_2} (n_1n_1',n_2n_2') = \sigma_{\nu_1, \nu_2} (n_1,n_2) \sigma_{\nu_1, \nu_2} (n_1',n_2').
\end{equation}
This means that it is enough to define this function for $n_1 = p^{k_1}, n_2 = p^{k_2}$. Put $\alpha := p^{\nu_1}$ and $\beta := p^{\nu_2}$. Then we define
\begin{equation}\label{eq:sigmaSchur}
	\sigma_{\nu_1, \nu_2} (p^{k_1}, p^{k_2}) = \beta^{-k_1} S_{k_1,k_2} (1, \beta, \alpha\beta),
\end{equation}
where 
\[
	S_{k_1,k_2}(\alpha, \beta, \gamma) = \left. \left| 
	\begin{matrix} \alpha^{k_1 + k_2 + 2} & \beta^{k_1 + k_2 + 2} & \gamma^{k_1 + k_2 + 2} \\ \alpha^{k_1 + 1} & \beta^{k_1 + 1} & \gamma^{k_1 + 1} \\ 1 & 1& 1 \end{matrix} \right| \middle/
	\left| \begin{matrix} \alpha^2 & \beta^2 & \gamma^2 \\ \alpha & \beta & \gamma \\ 1&1&1 \end{matrix} \right|\right.
\]
is the Schur polynomial. The Schur polnomials of Satake parameters give exactly the Fourier coefficients of $\GL_3$ automorphic forms at the powers of primes.

Let us establish the relationship with the triple divisor function.

\begin{lemma}
The Schur polynomial has the form,
\[
	\beta^{-k_1} S_{k_1,k_2} (1,\beta,\alpha \beta) = \sum_{i = 0}^{k_1} \sum_{j  = 0}^{k_2} \beta^{i + j} \alpha^{i} \sum_{\ell = 0}^{k_1 - i + j}  \alpha^{\ell},
\]
so that 
\[
	\sigma_{s_1, s_2}(n_1,n_2) = \sum_{e_1 | n_1} \sum_{e_2 | n_2} \sum_{e_3| \frac{n_1e_2}{e_1}} e_1^{s_1 + s_2} e_2^{s_2} e_3^{s_1}.
\]
\end{lemma}
This result follows from a straightforward calculation of the Schur polynomial and \eqref{eq:sigmaMultiplicative},\eqref{eq:sigmaSchur}.

	Substituting $n_1 = 1$ the above lemma simplifies as follows
\begin{corollary}
For $n \in \Z \backslash \{0\}$
	\begin{equation} \label{eq:divisorSumDefinition}
		\sigma_{s_1, s_2} (1, n) = \sum_{a | n} a^{s_2} \sum_{b|a} b^{s_1} = \sum_{n = e_1e_2e_3} e_1^{s_1 + s_2} e_2^{s_2},
	\end{equation}
	and in particular $d_3(n) = \sigma_{0,0}(1,n)$.
\end{corollary}

Now using the expression for the Kloosterman sum in Theorem \ref{thm:KloostermanEvaluation}, we write the Ramanujan sum. Compare with \cite[(6.3)]{BumpGL3}.
\begin{lemma}\label{lem:GeneralRamanujanSum}
	Given $c_1, c_2 \in \Z^{>0}$, let us call $R_{c_1, c_2}(n_1,n_2) = S_{w_0}(\vec{0}, \vec{n},; (c_1,c_2))$ the Ramanujan sum. Then,
	\[
		R_{c_1,c_2}(n_1,n_2) = \sum_{\substack{f | \gcd(c_1, c_2)\\ f | \frac{n_2 c_1}{f}}}  f c_{c_1/f}(n_1) c_f(n_2) c_{c_2/f}\left(\frac{c_1n_2}{f^2}\right).
	\]
\end{lemma}
\begin{proof}
	Simply by using \eqref{eq:FineCoarseKloostermanDecomposition} and Theorem \ref{thm:KloostermanEvaluation}, we write,
	\begin{align*}
		R_{c_1,c_2}(n_1,n_2) &= \sum_{f | \gcd(c_1,c_2)} f \S_{w_0}( \vec{0},\vec{n} ; \frac{c_1}{f}, \frac{c_2}{f}, f)\\
		&= \sum_{\substack{f | \gcd(c_1,c_2)\\ f | n_2d_1 }}  \sum_{ \substack{x_3 \pmod{f}\\ \gcd(x_3 , f) = 1} } f S\left(n_1, \frac{n_2d_1y_3}{f}, d_1\right) S\left(0,\frac{n_2 d_1}{f};d_2\right).
	\end{align*}
	Here $d_1 := \frac{c_1}{f}$ and $d_2:= \frac{c_2}{f}$. We can  evaluate the $y_3$ sum as
	\begin{align*}
		\sums_{y_3 \pmod{f}} S\left(n_1, \frac{n_2d_1y_3}{f},d_1\right) &=\sums_{u \pmod{d_1}} e\left(\frac{n_1u}{d_2}\right) \sums_{x_3 \pmod{f}} e\left(\frac{n_2 \overline{u} x_3}{f}\right) \\
		&= \sums_{u \pmod{d_1}} e\left(\frac{n_1u}{d_2}\right) c_{f} (m_2 \overline{u}) \\
		&= c_{d_2}(n_1) c_{f}(n_2) .
	\end{align*}
In the last line, we used the fact that $c_{f}(m_1\overline{u}) = c_{f}(m_1)$. We can do this because we have freedom to choose $\overline{u}$ as any element of the reduced residue classes $\pmod{d_2}$, so $\overline{u}$ can be a large prime, and in particular we can assume $\overline{u}$ is an integer relatively prime to $f$. This gives the result.
\end{proof}

Now using the same identity as Bump \cite{BumpGL3} we start to calculate the sum
\begin{equation*}
	\zeta(s_1) \zeta(s_2)  \zeta(s_1 + s_2 -1) \sum_{c_1, c_2>0 }\frac{R_{c_1,c_2} (n_1,n_2)}{c_1^{s_1} c_2^{s_2}} ,
\end{equation*}
in order to obtain $\sigma_{s_1,s_2}(n_1,n_2)$. Such equality can be justified via a study of Fourier coefficients $\SL_3$ Eisenstein series. Yet, this is an elementary statement expressing a divisor function as a double Dirichlet series of finite exponential sums. Discovering the form of the formula took us through $\SL_3$; however, as we see in the next proposition, an elementary proof can also be given.

\begin{proposition}
	For $\Re(s_1), \Re(s_2) > 1$, we have the identity
	\[
		\zeta(s_1 )\zeta(s_2) \zeta(s_1 + s_2 -1) \sum_{d_1, d_2 = 1}^\infty \frac{\mu(d_1)}{d_1^{s_1} d_2^{s_2}} \sum_{f | d_1n} \frac{c_f(n)c_{d_2}(\frac{nd_1}{f})}{f^{s_1 + s_2 -1}} = \sigma_{1 - s_1, 1- s_2}(1,n).  
	\]
	where $c_q(n)$ is the classical Ramanujan sum. 
\end{proposition}

\begin{remark} Compare this expression with the formula \cite[equations (6.3), (6.6)]{BumpGL3}, where $\Lambda_{B_1, B_2}(A_1, A_2)$ denotes the number of $C_1,C_2$ satisfying \eqref{eq:PluckerConditions} below for fixed $A_1, A_2, B_1, B_2$.\end{remark}

\begin{proof}
In  this proof we use the simplified notation $(a,b) = \gcd(a,b)$.

Substituting the form of the general Ramanujan sum from Lemma \ref{lem:GeneralRamanujanSum}, we start our calculation
	\begin{align*}
	\zeta(s_1) \zeta(s_2) & \zeta(s_1 + s_2 -1) \sum_{d_1, d_2=1}^\infty \frac{1}{d_1^{s_1} d_2^{s_2} }  \sum_{f | n_2d_1} \frac{c_{d_1}(n_1)c_f(n_2) c_{d_2}(n_2d_1/f) }{f^{s_1 + s_2 -1 }}\\
	=& \zeta(s_1) \zeta(s_1 + s_2 -1) \sum_{d_1=1}^\infty \frac{1}{d_1^{s_1}} \sum_{f | n_2d_1} \frac{c_{d_1}(n_1) c_{f}(n_2) \sigma_{1-s_2}(n_2d_1/f)}{f^{s_1 + s_2 - 1}}.
\end{align*}
Here we used the classical Ramanujan identity \eqref{eq:ramanujanFormula} on the $d_2$-sum. Let us assume $n_1 = 1$ now, so that $c_{d_1}(n_1) = \mu(d_1)$. Also put $(f,n_2) = e$. This gives $\gcd(n_2/e, f/e) = 1$, and so $\frac{f}{e} | d_1$. Changing variables $f/e \mapsto f$ we have,

\begin{align*}
	\zeta&(s_1)\zeta(s_1 + s_2 -1) \sum_{e | n_2} \frac{1}{e^{s_1+ s_2 - 1}} \sum_{d_1=1}^\infty \frac{\mu(d_1)}{d_1^{s_1}}  \sum_{\substack{f | d_1 \\ (f, n_2/e) = 1}} \frac{c_{fe} (n_2)  \sigma_{1-s_2}(\frac{n_2}{e}\frac{ d_1}{f})}{f^{s_1 +s_2 -1 }}   \\
	&= \zeta(s_1)\zeta(s_1 + s_2 -1) \sum_{e | n_2} \frac{1}{e^{s_1+ s_2 - 1}} \sum_{g | e} \mu(\tfrac{e}{g}) g \sum_{d_1=1}^\infty \frac{\mu(d_1)}{d_1^{s_1}}  \sum_{\substack{f | d_1 \\ (f, n_2/g) = 1}} \frac{\mu(f)\sigma_{1-s_2}(\frac{n_2}{e} \frac{ d_1}{f})}{f^{s_1 +s_2 -1 } } .
\end{align*}
Here we inserted $c_{q}(n) =\sum_{g |(q,n)} \mu(\frac{q}{g})g $, noting that $(fe, n_2) = e$. The coefficient of the $d_1$-Dirichlet series is almost a multiplicative function. We note that for a fixed $n$ the function $\sigma_\alpha(nd)/\sigma_\alpha(d)$ is a truly multiplicative function of $d$. Exchanging the order of the $e$ and $g$ sum we obtain,
\[
	\zeta(s_1)\zeta(s_1 + s_2 -1) \sum_{g | n_2} \frac{1}{g^{s_1+ s_2 - 2}} \sum_{e | n_2/g}\frac{\mu(e)\sigma_{1-s_2}(\frac{n_2/g}{e})}{e^{s_1 + s_2 -1}} \sum_{d_1=1}^\infty \frac{\mu(d_1)}{d_1^{s_1}}  \sum_{\substack{f | d_1 \\ (f, n_2/g) = 1}} \frac{\mu(f) \sigma_{1-s_2}(\frac{n_2/g}{e} \frac{d_1}{f})}{f^{s_1 +s_2 -1 }\sigma_{1-s_2}(\frac{n_2/g}{e})}  .
	\]
	
	For a cleaner notation we drop the subscripts at this point, and write $d, n$. Using the fact that the coefficients of the Dirichlet seris in the $d$-variable are multiplicative, this sum equals
	\begin{multline*}
		\zeta(s_1) \zeta(s_1 +s _2- 1) \sum_{g | n} \frac{1}{g^{s_1 + s_2 - 2} } \sum_{e | n/g} \frac{\mu(e)\sigma_{1-s_2}(\frac{n/g}{e})}{e^{s_1 + s_2 - 1}} \\
		\times \prod_{q \nmid \frac{n}{g} } \left( 1 - \frac{1}{q^{s_1}}\left( \sigma_{1 - s_2} (q) - \frac{1}{q^{s_1 + s_2 - 1}} \right) \right)
		\prod_{p | \frac{n}{g} } \left( 1 -  \frac{1}{p^{s_1} } \frac{\sigma_{1-s_2}(\frac{n/g}{e} p)}{\sigma_{1-s_2}(\frac{n/g}{e})}\right) .
	\end{multline*}
	The $q$ factor is
	\[
		1 - \frac{1}{q^{s_1}} - \frac{1}{q^{s_1 + s_2 -1 }} + \frac{1}{q^{2s_1 + s_2 -1}} = \left(1 - \frac{1}{q^{s_1}}\right) \left(1 - \frac{1}{q^{s_1 + s_2 -1}}\right),
	\]
	which cancel with the Euler factors of the two zeta functions.	
	
	So let us assume $n = p^k$. If $g = n$ then the $e$ sum is simply $1 = \sigma_{1-s_2}(n/g)$. Now if $g \neq n$, we have the $e$ sum as,
	\[
	\left(\sigma_{1- s_2}(n/g) -\frac{\sigma_{1-s_2}(pn/g)}{p^{s_1}}- \frac{\sigma_{1-s_2}(\frac{n/g}{p})}{p^{s_1 + s_2 -1 }} + \frac{\sigma_{1-s_2}(n/g)}{p^{2s_1 + s_2 -1}} \right).
	\]
	We then apply the Hecke relation for divisor sums, i.e. that if $p|n$,
	\[
		\sigma_{\alpha}(np) = \sigma_\alpha (n) \sigma_\alpha(p) - p^\alpha \sigma_\alpha(n/p).
	\]
	Thus we have 
\begin{multline*}	
 \left(\sigma_{1- s_2}(n/g) -\frac{\sigma_{1-s_2}(n/g)}{p^{s_1}} \sigma_{1-s_2}(p) + \frac{\sigma_{1-s_2}(\frac{n/g}{p})}{p^{s_1 + s_2 - 1}}  - \frac{\sigma_{1-s_2}(\frac{n/g}{p})}{p^{s_1 + s_2 -1 }} + \frac{\sigma_{1-s_2}(n/g)}{p^{2s_1 + s_2 -1}} \right)
\\
	= \sigma_{1-s_2}(n/g) \left(1 - \frac{1}{p^{s_1}}\left(1  + \frac{1}{p^{s_2 - 1}} \right) + \frac{1}{p^{2s_1 +  s_2 -1}}\right) \\= \zeta_p(s_1) \zeta_p(s_1 + s_2 - 1) \sigma_{1-s_2}(n/g).
\end{multline*}
Here $\zeta_p(s) = (1 - p^{-s})\inv$, is the $p$-Euler factor, that cancels with the Riemann zeta function.

Therefore we obtain that the whole sum is,
$
 \sum_{g | n } \frac{1}{g^{s_1 + s_2 - 2}} \sigma_{1 - s_2}(n/g) .
$	
\end{proof}

\section{Pl\"ucker Coordinates}\label{sec:correspondence}

At this point we would like to emphasize one subtlety about the set of representatives for the double cosets
\[
	\Gamma_\infty \backslash Bw_0 B \cap \SL_3(\Z)/ \Gamma_\infty,
\]
given in  \cite[Proposition 3.13]{BumpFriedbergGoldfeld}.

There these double cosets are parametrized by sextuples $(A_1, B_1, C_1, A_2, B_2, C_2)$ with $A_1, A_2 >0$ and $B_1,C_1 \pmod{A_1}, B_2, C_2 \pmod{A_2}$ satisfying 
\begin{equation}\label{eq:PluckerConditions}
A_1C_2 + B_1B_2 + C_1A_2 = 0 \text{ and } \gcd(A_1, B_1, C_1) = \gcd(A_2, B_2, C_2) = 1.
\end{equation}
Then the set of representatives are given as
\[
	R_{w_0} = \left\{ \bpm a_{11} & a_{12} & a_{13} \\ a_{21} & a_{22} & a_{23} \\ A_1 & B_1 & C_1 \epm \right\},
\]
such that the minors are
\[
	A_2 = a_{21} B_1 - a_{22} A_1 \qquad B_2 = -(a_{21}C_1 - a_{23} A_1) \qquad C_2 = a_{22} C_1 - a_{23}B_1.
\]
and the determinant of the matrix is $1$. 

After such a decomposition the quantities $A_1$ and $A_2$, the successive minors from the lower left corner, fix the diagonal entries in the Bruhat decomposition, i.e.
\[
	\bpm a_{11} & a_{12} & a_{13} \\ a_{21} & a_{22} & a_{23} \\ A_1 & B_1 & C_1 \epm = u_L w_0 \bpm A_1 & & \\ &A_2/A_1&\\ &&1/A_2\epm u_R.
\]
Here $u_L$ and $u_R$ are unipotent matrices given in Proposition 3.7 ibid.

The subtle point is about the ordering of the variables, which is made clear in \cite[Remark 3, p. 302]{BBFHWeylMDSIIIAnnals}. It is crucial to first choose $B_1 \pmod{A_1}$ and $B_2 \pmod{A_2}$ and fix them as integers (as opposed to residue classes modulo $A_1$ and $A_2$ respectively) before choosing $C_1 \pmod{A_1}$ and $C_2\pmod{A_2}$. 

If the integers $B_1,C_1, B_2, C_2$ are chosen simultaneously and then filtered according to the conditions \eqref{eq:PluckerConditions}, then one can find distinct sextuples that correspond to the same $\Gamma_\infty$--double coset. For example $(2,1,0,2,2,-1)$ and $(2,1,1,2,0,-1)$ as distinct sextuples, but the $3\times 3$ matrices they give rise to are in the same double coset.

\section{Bound on long word Kloosterman sums} \label{sec:KloostermanBound}

In this section we will use the notation $(a,b)$ in place of $\gcd(a,b)$ to make the following less cumbersome. We hope that in the following formulas the distinction between a tuple and the greatest common divisor of two integers is clear.

Stevens, in \cite{Stevens}, has bounded these long word Kloosterman sums as
\begin{equation}\label{eq:stevensBound}
	|S_{w_0}(\vec{m}, \vec{n}; (c_1,c_2) ) | \leq \tau(c_1) \tau(c_2) (m_1 n_2, C)^{\frac12} (m_2n_1, C)^{\frac 12} (c_1,c_2)^{\frac12} \sqrt{c_1c_2},
\end{equation}
where $C = \operatorname{lcm}(c_1,c_2)$.
See \cite[Theorem 4]{ButtcaneSumOfSL3Kloosterman} for the above formulation.

As can be seen with the exact calculation \eqref{eq:KloostermanSumEvaluationExample} the end of Subsection \ref{subsec:FineKloostermanEvaluation}, the bound
\[
	S_{w_0}((1,p), (1,p); (p^2,p)) =  O\big(p^{5/2}\big)
\]
is sharp. The bound \eqref{eq:stevensBound}, on the other hand, would imply an upper bound on the order of $O_\epsilon(p^{3 + \varepsilon})$.

Given the decomposition of Kloosterman sum as a sum of Kloosterman sums as in Theorem \ref{thm:KloostermanEvaluation}, using the Weyl bound on the individual sums,
\begin{equation}\label{eq:standardRWDbound}
\begin{split}
	|S_{w_0}(\vec{m}, \vec{n}; &(c_1, c_2))| \\
	&\leq \sum_{f| (c_1, c_2) } f \sum_{\substack{ x_3 y_3 \equiv 1 \pmod{f}\\ m_2d_2 + y_3 n_2d_1 \equiv 0 \pmod{f} }} (n_1, d_1)^{\frac12} (m_1,d_2)^{\frac{1}{2}} \sqrt{d_1d_2} \tau(d_1) \tau(d_2) \\
	&\leq \sum_{\substack{f| (c_1,c_2)\\ (m_2d_2,f) = (n_2d_1,f)}} (f, m_2d_2)  (n_1, d_1)^{\frac12} (m_1,d_2)^{\frac{1}{2}} \sqrt{c_1 c_2} \tau(d_1)  \tau(d_2).
\end{split}
\end{equation}
Here $d_i = c_i/f$ and we bounded the number of solutions
to the congruence equation $m_2d_2 + y_3 n_2 d_1 \equiv 0 \pmod{f}$ with $y_3 \in (\Z/f\Z)^*$ by simply $(m_2d_2,f)$.

Notice that this answer is not symmetric in the variables. However the decomposition of $S_{w}$ into the stratification induced by $w_0 = s_\beta s_\alpha s_\beta$, as in Section \ref{subsec:BraidRelation} comes to the rescue, and thus we also have from  Proposition \ref{prop:braidRelation} that
\begin{equation} \label{eq:braidBound}
	S_{w_0} (\vec{m}, \vec{n}; (c_1,c_2)) \leq  \hspace{-0.3cm}\sum_{\substack{f | (c_1,c_2)\\ (n_1d_2,f) = (m_1d_1,f)}}  \hspace{-0.3cm} (f,m_1d_1) (m_2,d_1)^{\frac 12} (n_2,d_2)^{\frac 12} \sqrt{c_1c_2}\tau(d_1) \tau(d_2).
\end{equation}

In order to show that our result is at least as strong as \eqref{eq:stevensBound}, we combine the above two results and obtain, in general, the following proposition.

\begin{proposition}
	Given $\vec{m}, \vec{n} \in \Z^2 -(0,0)$, and $c_1,c_2>0$, we may bound the long word \emph{coarse} Kloosterman sum as 
	\[
		|S_{w_0}(\vec{m},\vec{n}; \vec{c} )| \leq \sqrt{c_1c_2} (c_1,c_2)^{\frac12} \tau((c_1,c_2)) \tau(c_1) \tau(c_2) \min\{A, B\},
	\]
	where $\tau(c)$ is the number of divisors of $c$ and
	\begin{align*}
		A &= (m_2n_1, c_1)^{\frac 12}  (n_2m_1, c_2)^{\frac 12 },\\
		B &= (m_2n_1, c_2)^{\frac 12} (n_2m_1, c_1)^{\frac 12} .
	\end{align*}  
\end{proposition}
\begin{proof}
	Since we are adding over $f$ such that $(m_2d_2,f) = (n_2d_1,f)$, we write in \eqref{eq:standardRWDbound}, 
	\[
		(f, m_2d_2)^2 = (f,m_2d_2) (f, n_2d_1) = (f, d_2)(f,d_1) (\tfrac{f}{(f,d_2)}, m_2)(\tfrac{f}{(f,d_1)}, n_2)) .   
	\]
	Combining this with $(\frac{f}{(d_2,f)}, m_2) (d_1, n_1) \leq (d_1f, m_2n_1) = (c_1,m_2n_1)$, and similarly with $(\frac{f}{(d_1,f)}, n_2)(d_2, m_1) \leq (c_2, m_1n_2)$ we get the term $\sqrt{(d_1,f)(d_2,f)} A$. Assume that $c_1 = p^k$ and $c_2 = p^\ell$ with $\ell \leq k$. Then as $f$ runs through powers of $p$, the maximum value of $(d_1,f)(d_2,f)$ is achieved for $f = p^r$ with $ \frac{\ell}{2} \leq r \leq \frac{k}{2}$ and that value is $< p^{\ell}$. By multiplicativity we get that, 
	\[
	\max_{f| (c_1, c_2) }(\tfrac{c_1}{f},f) (\tfrac{c_2}{f}, f) \leq (c_1,c_2). 
\]
  There are at most $\tau((c_1,c_2))$ many summands. This gives us the bound with $A$. Starting with \eqref{eq:braidBound} instead, we get the bound with $B$. Considered together, we obtain the given statement.
\end{proof}

Notice that this is still stronger than \eqref{eq:stevensBound} in its $\vec{m}$ and $\vec{n}$ dependence and only weaker in its $\vec{c}$ dependence by a very small factor of $\tau((c_1,c_2))$. This is despite the fact that in the above proof we used many potentially not sharp inequalities. 

We urge the reader who may need a specific sharp upper bound propositions to work directly with \eqref{eq:standardRWDbound} and \eqref{eq:braidBound} instead of the above proposition, since their specific situation may significantly reduce the bounds. For example if $c_1$ and $c_2$ are squarefree, then we may get rid of the $(c_1,c_2)^{\frac12}$ factor. Also if $m_2 = 1$--which can be a common occurence if these Kloosterman sums are obtained from analyzing moments of $\GL_3$ Maass form $L$-functions--then the choice of $f$ is severely restricted. For example in the fine Kloosterman sum decomposition of 
\[
	S_{w_0}((m_1, 1), (n_1,n_2); (p^k, p^\ell)) ,
\]
with $\ell< k$, the condition $(m_2d_2,f) = (n_2d_1,f)$ restricts the allowed $f$ to $f = p^{r}$ with $r \leq \lfloor \frac{\ell}{2} \rfloor$. 

\section{Factorization of other Bruhat Cells}

If $w = s_\alpha$ notice that the associated Bruhat cell is given as 
\[
	Bs_\alpha B \cap \SL_3(\Z) = \{A \in \SL_3(\Z): A_{31} = A_{32} = 0, A_{21} \neq 0 \}.
\]
We cannot hope to write every such element in the form $\iota_\alpha(\gamma_1)$, however we do not need to do it for every matrix in the Bruhat cell, but for only one representative in the $\Gamma_\infty$--double coset. Note that the Kloosterman sum is well defined only if \eqref{eq:consistencyConditions} are satisfied. Also a quick calculation shows that $c_2 = 1$. These allows us to write any $\gamma \in \Omega_{s_\alpha}(d,1)$ as $\gamma = u_L s_\alpha t(d,1) u_R$ with $u_R \in \U^{s_\alpha}$. After some calculation we see that
\[
	\left\{\iota_{\alpha}\left(\bsm x & \frac{xy - 1}{d}\\ d& y \esm \right)\middle | x,y \pmod d, xy \equiv 1 \pmod d\right\}
\] 
form a complete set of representatives for $\Gamma_\infty \backslash \Omega_{s_\alpha}(d, 1) /\Gamma_\infty$.

Similar results are true for the Weyl group elements of length 2. Let us first describe the Bruhat decompositions of such cells. We will be as terse as possible, only stating the relevant results in our notation, since these Kloosterman sums have been studied in the general case by Friedberg \cite{Friedberg1987}.

\begin{proposition}
	Given $w = s_\alpha s_\beta = \bsm &&1\\ 1&&\\ &1& \esm$, the Bruhat cell is defined as
	\[
		BwB \cap \SL_3(\Z) = \left\{A \in \SL_3(\Z) : A_{31} = 0, A_{21} \neq 0,A_{32}\neq 0 \right\},
	\]
	and the Bruhat decomposition of any such $A$ is given via the coordinates,
	\[
		A = \bpm 1& \frac{\langle \vec{e}_1^*, A\vec{e}_1\rangle }{t_1} & a_{13} \\ & 1& a_{23} \\ &&1 \epm s_\alpha s_\beta  t \bpm 1&b_{12}&b_{13} \\ &1& \frac{\langle \vec{e}_{3}^*, A\vec{e}_{3}\rangle  }{t_2} \\ &&1 \epm  .
	\]
	Here $t = \diag(t_1,t_2,t_3)$ with $t_1 = A_{21}$, $t_2 = A_{32}$ (equivalently $t_1t_2= M_{13}$) and $t_3 = 1/(t_1t_2)$. The coordinates $a_{23}$ and $b_{12}$ are tied to one another via the equation $A_{22} = \langle \vec{e}_{2}^*, A\vec{e}_{2} \rangle = t_1b_{12} + t_2a_{23}$. Once a choice of $a_{23}$ (or $b_{12}$) has been made this determines $a_{13}, b_{13}$ via
	\[
		a_{13} = \frac{A_{11}}{t_1}a_{23} - M_{33} \qquad \text{ and } \qquad b_{13} = b_{12} \frac{A_{33}}{t_2} - M_{11}.
	\]
\end{proposition}
	The fact that there is a one parameter freedom in our Bruhat decomposition parameters is due to $\U_{s_\alpha s_\beta}$ being a one dimensional group. 
	
	We now factorize the elements in this group.

\begin{proposition}
Let $w = s_\alpha s_\beta = \bsm &&1\\ 1&&\\&1&\esm$.  We have the decomposition
\[
	BwB \cap \SL_3(\Z) = \bigcup_{\substack{(d_1,d_2) \in \Z^2\\d_1d_2 \neq 0}} \Omega_{w}(d_1,d_1d_2),
\]
and for every element of $\Omega_{w}(d_1,d_1 d_2)$ there is a $\Gamma_\infty$--double-coset representative that can be factored as $\iota_{\alpha}(\gamma_1) \iota_\beta(\gamma_2) $, where
\[
	\gamma_1 = \bpm x_1 & b_1\\ d_1 & y_1 \epm, \quad \gamma_2 = \bpm x_2 & b_2\\ d_2 & y_2 \epm
\]
are two $\SL_2$ matrices.

Finally we also note that 
\[
	BwB \cap \Gamma_0(N) = \bigcup_{\substack{d_1, d_2 \in \Z\\ N| d_2, d_1d_2 \neq 0}} \Omega_{w}(d_1,d_1d_2).
\]
\end{proposition}
\begin{proof}

We calculate that the effect of $\iota_\alpha(\gamma_1)\iota_\beta(\gamma_2)$ on standard basis elements of the exterior algebra,
\begin{align*}
	&&\vec{e}_{1} \mapsto x_1 \vec{e}_1 + d_1 \vec{e}_{2} && \vec{e}_{1,2} \mapsto x_2 \vec{e}_{1,2} + x_1 d_2 \vec{e}_{1,3} + d_1d_2 \vec{e}_{2,3}\\
	&&\vec{e}_{2} \mapsto x_2 b_1 \vec{e}_{1} + x_2 y_1 \vec{e}_{2} + d_2 \vec{e}_{3} && 
	\vec{e}_{1,3} \mapsto b_2 \vec{e}_{1,2} + y_2 x_1 \vec{e}_{1,3} + y_2 d_1 \vec{e}_{2,3} \\
	&&\vec{e}_{3} \mapsto b_1 b_2 \vec{e}_{1} + b_2 y_1 \vec{e}_{2} + y_2 \vec{e}_{3} && \vec{e}_{2,3} \mapsto b_1 \vec{e}_{1,3} + y_1 \vec{e}_{2,3}.
\end{align*}
There are two identities, $\langle \vec{e}_{3}^*, A\vec{e}_{1} \rangle = 0 $ and $\langle \vec{e}_{1,2}^*, A\vec{e}_{2,3} \rangle = 0$. The first one is satisfied by all the matrices in the Bruhat cell, whereas for the latter, i.e. for the minor $M_{31}$ to vanish, we may need to pass to possibly another $\Gamma_\infty$ representative. From the determinant condition we have that $- A_{23} M_{23} + A_{33}M_{33} = 1$, in particular $\gcd(M_{23},M_{33}) = 1$. This means that by multiplying with a suitable $\bsm 1& n & m \\ &1&  0 \\ && 1 \esm$ with $n,m \in \Z$ on the right, we can make sure $M_{31} = 0$.

Now that there is a hope of such factorization, we begin with the Ansatz that $A = \iota_\alpha(\gamma_1) \iota_\beta(\gamma_2)$ and calculate the consequences. Reading off the above table, we get that
\[
	\gamma_1 = \bpm A_{11} & M_{21} \\ A_{21} & M_{11} \epm \qquad \text{ and } \qquad 
	\gamma_2 = \bpm M_{33} & M_{32} \\ A_{32} &A_{33}  \epm. 
\]
Note that $\gamma_1, \gamma_2 \in \SL_2(\Z)$. Multiplying under the assumption that $\det(A) = 1$ and $M_{31} = 0$, gives us $A$ back. 

Finally we note that $A \in \Gamma_0(N)$ if and only if $N | A_{32}= d_2$, which gives the final statement of the proposition.
\end{proof}

From the $\det(A) = 1$ condition we have $A_{33}M_{33} \equiv 1 \pmod{A_{32}}$ and in the $\Gamma_\infty$--double cosets the elements $A_{33}, M_{33}$ are determined up to modulo $A_{32}$. Similarly $A_{11}$ and $M_{11}$ are determined up to modulo $A_{21}$ and expanding the determinant of $A$ along the first column, we get $A_{11} M_{11} \equiv 1 \pmod{A_{21}}$. The compatibility relation for $w$ is, $n_1 = d_1m_2/d_2$. We also have $A_{22} = M_{11} M_{33}$ given the assumption $A_{33}$ . The Kloosterman sum associated to $w = s_\alpha s_\beta$ is 
\[
	S_{s_\alpha s_\beta}(m_1,n_1,n_2; d_1,d_1d_2) =  \sideset{}{^*}\sum_{\substack{x_1 \pmod{d_1}\\ x_2 \pmod{d_2}}} e\left(m_1\frac{x_1}{d_1} + n_1\frac{\overline{x_1} x_2 }{d_1} + n_2\frac{\overline{x_2}}{d_2}\right).
\]
We use the same notation for this hyperkloosterman sum as in the literature, see \cite{BumpFriedbergGoldfeld, Friedberg1987, BlomerButtcaneMaga}, however we would like to make a point. This notation is not ideal, because it hides the condition $d_2n_1 = d_1m_2$; i.e. that $m_2$ is a multiple of $\gcd(d_1,d_2)$. Of course in the Bruggeman Kuznetsov formula this condition is always imposed. Let us note that the sum is well defined, and does not depend on a choice of $\overline{x_2} \pmod{d_2}$, since we are able to write it also in the form $ \sideset{}{^*}\sum_{\substack{x_1 \pmod{d_1}\\ x_2 \pmod{d_2}}} e\left(m_1\frac{x_1}{d_1} + m_2\frac{\overline{x_1} x_2 }{d_2} + n_2\frac{\overline{x_2}}{d_2}\right)$.

For the Weyl group element $w = s_\beta s_\alpha = \bsm &-1&\\ &&-1\\1&& \esm$ factorizing elements of $BwB \cap \SL_3(\Z)$ can be achieved using the same methods. We may also instead use the homomorphism $A \mapsto A^\dagger$ introduced in Section \ref{subsec:BraidRelation}. Then we can immediately say that this cell is determined by the equations $M_{13} = 0, M_{23} \neq 0$ and $M_{12} \neq 0$. Furhthermore given any $A \in BwB \cap \SL_3(\Z)$ satisfying $A_{13} = 0$ can be written as a product $\iota_\beta(\gamma_1) \iota_\alpha(\gamma_2)$ where we explicitly have
\[
	\gamma_1 = \bpm M_{33} & A_{23} \\ M_{23} & A_{33} \epm \qquad \text{ and } \qquad \gamma_2 = \bpm A_{11} & A_{12} \\ M_{12} & M_{11} \epm.  
\]
 The condition $A_{13}$ can be achieved since the determinant condition implies $A_{11}M_{11} - A_{12} M_{12} = 1$, and in particular $\gcd(A_{11}, A_{12}) = 1$. Thus multiplying by a certain $\Gamma_\infty$ element on the right we can obtain $A_{13} = 0$. 
 
 The Kloosterman sum associated to this Weyl group element is given as
 \[
  	S_{s_\beta s_\alpha } (m_1,m_2,n_1; d_1d_2, d_1)= 
  	\sideset{}{^*}\sum_{\substack{x_1 \pmod{d_1}\\ x_2 \pmod{d_2}}} e\left(m_1\frac{x_2 \overline{x_1}}{d_2}  + m_2\frac{x_1}{d_1} + n_1 \frac{\overline{x_2}}{d_2}\right), 
 \]
 with the condition that $d_2 n_2 = d_1 m_1$.  This can be also written in terms of the previous hyperkloosterman sum: $S_{s_\beta s_\alpha}(m_1,m_2,n_1;d_1d_2,d_1) = S_{s_\alpha s_\beta}(m_2, n_2, n_1; d_1,d_1d_2)$.

The $\Gamma_0(N)$ condition, is equivalent to $N| M_{23} = d_1$. On the one hand that $\gcd(A_{31}, A_{32} ) | M_{23}$ is clear. On the other hand, since $M_{13} = 0$, we have $A_{31} = M_{12}M_{23}$ and $A_{32} = M_{11}M_{23}$. Opening the determinant along the first row, gives us that $\gcd(M_{11}, M_{12}) = 1$, and thus we get that $M_{23} = \gcd(A_{31}, A_{32})$.

\section{Kuznetsov Trace Formula} \label{sec:Kuznetsov}

For the benefit of the reader we will write down the Kuznetsov trace formula using our parametrizations.

The Bruggeman-Kuznetsov Trace formula for $\SL_3$ has been written down in many sources, such as in \cite{LiSpectralMeanValue} which is specialized from the statement of the general $\GL_n$--Bruggeman-Kuznetsov formula in \cite[Chapter 11.6]{GoldfeldGLnRBook} (computed again by Xiaoqing Li, as stated in the beginning of that section), or also in \cite{buttcaneSpectral}. We will however follow the notation of \cite{BlomerButtcaneMaga}, as they also cover the congruence subgroup case.

Here $y = \diag(y_1y_2, y_1, 1)$.

\begin{theorem}
	Let $N, n_1,n_2,m_1, m_2 \in \mathbb{N}^{>0}$, and $F: (0,\infty)^2 \to \C$, a smooth compactly supported test function. Using the notation of \cite[Theorem 6]{BlomerButtcaneMaga}, the Bruggeman-Kuznetsov formula for the theorem may be written as 
	\[
		\int_{(N)} 	\frac{A_{\varpi}(n_1, n_2) \overline{A_{\varpi} (n_1,n_2)} }{\mathcal{N}(\varpi)} \left| \langle \tilde{W}_{\mu_\pi}, F \rangle  \right|^2 \d \varpi= \Delta + \Sigma_4 + \Sigma_5 + \Sigma_6 ,
	\]
	with
	\begin{align*}
		\Delta &= \delta_{n_1, m_1} \delta_{n_2 ,m_2} \|F \|^2 ,\\
		\Sigma_4 &= \sum_{\varepsilon = \pm 1} \sum_{\substack{d_1,d_2 > 0\\ d_2n_1 = d_1 m_2\\ N|d_2}} S_{s_\alpha s_\beta}(\varepsilon m_1, n_1, n_2;d_1,d_1d_2) \tilde{\mathcal{J}}_{\varepsilon, F} \left(\frac{\sqrt{n_1n_2m_1}}{d_1 \sqrt{d_2}} \right) ,\\
		\Sigma_5 &= \sum_{\varepsilon = \pm 1} \sum_{\substack{d_1,d_2 > 0\\ d_2n_2 = d_1m_1\\ N| d_1}} S_{s_\alpha s_\beta}(\varepsilon m_2,  n_2, n_1; d_1,d_1d_2) \tilde{\mathcal{J}}_{\varepsilon, F^*} \left(\frac{\sqrt{n_1n_2m_2}}{d_1 \sqrt{d_2}} \right), 
		\\
		\Sigma_{6} &= \sum_{\vec{\varepsilon} \in \{\pm 1\}^2} \sum_{\substack{d_1,d_2,f>0\\ N|f}} \S_{w_0}(\vec{m}^{\varepsilon}, \vec{n}; d_1,d_2,f)   \mathcal{J}_{\varepsilon,F} \left(\tfrac{\sqrt{n_2 m_1 d_1 }}{d_2 \sqrt{f}}, \tfrac{\sqrt{n_1 m_2 d_2}}{d_1 \sqrt{f}} \right).
	\end{align*}
	Here $\delta_{n,m}$ is the Kronekcer-$\delta$ function, $\vec{m}^\varepsilon = (\varepsilon_1 m_1, \varepsilon_2 m_2)$, and $F^*(y_1, y_2) = F(y_2,y_1)$.
\end{theorem}
Here the higher rank analogues of the Bessel functions have been explicitly calculated as in \cite[(2.4), (2.5)]{BlomerButtcaneMaga}.
The Kloosterman sums relating to the long word element are weighted by the functions
\begin{multline*}
	\mathcal{J}_{\vec{\varepsilon}}(A_1, A_2) = \frac{1}{(A_1A_2)^2} \int\limits_{0}^\infty \int\limits_{0}^\infty \int\limits_{-\infty}^\infty \int\limits_{-\infty}^\infty \int\limits_{-\infty}^\infty e (-\varepsilon_1 A_1 x_1 y_1 - \varepsilon_2 A_2 x_2 y_2 ) \\
	\times e\left(-\frac{A_{2}}{y_2} \frac{x_1x_3 + x_2}{x_3^2 + x_2^2 + 1}\right) e\left(-\frac{A_1}{y_1} \frac{x_2(x_1 x_2 - x_3) + x_1}{(x_1x_2 -x_3)^2 + x_1^2 + 1}\right)
\overline{F(A_1 y_1 , A_2 y_2 )}\\
F\left(\frac{A_2}{y_2}\frac{\sqrt{(x_1 x_2 - x_3)^2  + x_1^2 + 1}}{x_3^2 + x_2^2 + 1},\frac{A_1}{y_1}\frac{\sqrt{x_3^2 + x_2^2 + 1}}{(x_1x_2 - x_3)^2 + x_1^2 + 1}\right) \d x_1 \d x_2 \d x_3 \frac{\d y_1 \d y_2}{y_1 y_2},
\end{multline*}
and the transform for the cyclical Weyl group elements are given by
\begin{multline*}
	\tilde{\mathcal{J}}_{\varepsilon, F^*} (A) = \frac{1}{A^2} \int\limits_{0}^\infty \int\limits_{0}^\infty \int\limits_{-\infty}^{\infty} \int\limits_{-\infty}^{\infty} e(−\varepsilon A x_1 y_1 )
	e\left(y_2 \frac{x_1 x_2}{x_1^2 + 1}\right) 
	e\left(\frac{A}{y_1y_2}\frac{x_2}{x_1^2 + x_2^2 + 1}\right)\\
	\times F\left(y_2 \frac{\sqrt{x_1^2 + x_2^2 + 1} }{x_1^2 + 1}, \frac{A}{y_1y_2}\frac{\sqrt{x_1^2 + 1}}{x_1^2 + x_2^2 + 1}\right)\overline{F(Ay_1, y_2)} \d x_1 \d x_2 \frac{\d y_1 \d y_2}{y_1 y_2^2}.
\end{multline*}
 
For the exact definitions of the terms on the spectral side, see \cite{BlomerButtcaneMaga}.

In the above theorem the authors Blomer, Buttcane and Maga have chosen to denote by $\int_{(N)} d̟\varpi$ a combined sum/integral
over the complete spectrum of level $N$. This is a rather terse notation, which goes against the spirit of this paper. However, if we were to write the spectral side of this formula explicitly, we would have had to introduce more automorphic notation than that is necessary considering the scope of this present paper. Furthermore we should add that writing a completely explicit spectral side of the $\Gamma_0(N)$--Bruggeman-Kuznetsov trace formula--akin to the $\SL_2$ case--is not a trivial task. If the reader would like to syntesize the explicit form of the left hand side from the literature for themselves, we direct them to the thesis of Balakci \cite{balakci}, where the ``\emph{cusps}'' of maximal parabolic subgroups of $\Gamma_0(N) \subseteq \SL_3(\Z)$ are parametrized.

\bibliographystyle{alpha}
\bibliography{biblio}

\end{document}